\documentclass[envcountsect,envcountsame]{svproc}

\PassOptionsToPackage{x11names}{xcolor}
\usepackage{xcolor}
\usepackage{graphicx} 
\usepackage{amssymb}
\usepackage{pgf}
\usepackage{tikz}
\usetikzlibrary{fadings, decorations.pathreplacing}\usetikzlibrary{arrows.meta}
\usepackage{multicol}
\usepackage{mathtools}
\usepackage{setspace}
\usepackage[T1]{fontenc}
\usepackage[percent]{overpic}

\usepackage{bbm}
\usepackage{cleveref}

\usetikzlibrary{arrows,automata}

\newcommand{\NN}{\mathbb{N}}
\newcommand{\RR}{\mathbb{R}}
\newcommand{\CC}{\mathbb{C}}

\usepackage{tikz-network}

\newcommand{\p}{{\raisebox{1.3pt}{{$\scriptscriptstyle\bullet$}}}}
\newcommand{\Tt}{(T(t))_{t\ge0}}

\DeclareMathOperator{\diag}{diag}

\usepackage{changes}
\colorlet{change}{red}
\colorlet{draft}{blue}

\newtheorem{thm}{Theorem}[subsection]
\newtheorem{lem}[thm]{Lemma}
\newtheorem{prop}[thm]{Proposition}
\newtheorem{rem}[thm]{Remark}

\definechangesauthor[name={Change}, color=change]{Change}

\usepackage{url}

\begin{document}
\mainmatter     
\title{Journey Through the World \\ of Dynamical Systems on Networks}
\titlerunning{Journey Through the World of Dynamical Systems on Networks}  
%
\author{M. N. Cartier van Dissel\inst{1} \and P. Gora\inst{2} \and
M. Iskrzyński\inst{3} \and M. Kramar Fijav\v{z}\inst{4} \and D. Manea\inst{5} \and A. Mauroy\inst{6} \and I. Naki\'c\inst{7} \and S. Nicaise\inst{8} \and M. B. Paradowski\inst{9} \and A. Puchalska\inst{2} \and G. Rotundo\inst{11} \and E. Sikolya\inst{12}}
\authorrunning{Puchalska et al.} 
%
%
\institute{Complexity Science Hub, Vienna, Austria
\and
Institute of Applied Mathematics and Mechanics,\\ University of Warsaw, Warsaw, Poland
\and
Systems Research Institute, Polish Academy of Sciences, Warsaw, Poland
\and
Faculty of Civil and Geodetic Engineering, University of Ljubljana / Institute of Mathematics, Physics, and~Mechanics,  Ljubljana, Slovenia
\and
Institute of Mathematics, Romanian Academy of Sciences, Bucharest, Romania
\and
Department of Mathematics and Namur Institute for Complex Systems (naXys), University of Namur, Namur, Belgium
\and
Department of Mathematics, Faculty of Science, \\ University of Zagreb, Zagreb, Croatia
\and
Université Polytechnique Hauts-de-France, Valenciennes, France
\and
Institute of Applied Linguistics, University of Warsaw, Warsaw, Poland
\and
Department of Statistical Sciences, Sapienza University of Rome, Rome, Italy
\and
Department of Applied Analysis and Computational Mathematics, \protect\\ E\"otv\"os Lor\'and University, Budapest, Hungary
}

\maketitle              

\begin{abstract}
We present a subjective selection of methods for complex systems analysis ranging from statistical tools through numerical methods based on AI to both linear and non-linear ODEs and PDEs. All the notions apply the network structure and are presented in the context of applied problems to visualise the strengths and drawbacks of the approach. The major aim of capturing such a broad overview is to understand the interrelations between network theories that  seem to be distant from the mathematical perspective.
\end{abstract}

 \setcounter{tocdepth}{1}
 \tableofcontents

\section{Introduction}

In the modelling of real-life phenomena, the first challenge that one needs to face is the choice of appropriate mathematical tools. Aiming for accuracy, researchers strive to include essential information about the complex phenomenon while also eliminating details that have only minor effects on the whole process.

Especially in fields where social and biological aspects play a role, individual variability, an abundance of interconnected sub-components of the phenomena and uncertainty about their actual interactions render mathematisation of these branches of science difficult. 
This can lead to a situation where the choice of methods seems to lack a clear basis, with the modeller's area of expertise sometimes playing a larger role than any specific, underlying rationale. The question naturally arises of the interplay between different techniques and lies at the core of mathematical modelling.

In response to this issue, we propose to examine a spectrum of open problems in a wide range of topics such as language studies, finances, economy, ecology, epidemiology, and traffic flow. All these questions are united by the need to characterise complex dynamics at different levels of description. Consequently, a wide range of mathematical methods can be applied ranging from statistical analysis through numerical methods based on AI to both linear and non-linear ODEs and PDEs. All of them, however, developed their own mathematical language to incorporate the notion of a network that permits clustering entities that join certain dynamical properties on the one hand, and stating interdependence on the other. 

The authors do not claim to present an optimal methodology for dealing with the presented open problems, or to give a comprehensive summary of the toolbox for dealing with problems of dynamical systems on networks. 
However, by presenting a range of methods that have proven useful in specific applied problems outlined in Section \ref{sec:problems}, the author aims to highlight the wide array of methods available for application. These are divided based on the type of network that lies at the core of the model and can be assigned to one of four categories (defined formally in Section \ref{sec:prelimi}), namely: (combinatorial) 
graphs (Section \ref{sec:combinat}), graphs with dynamics on vertices (Section \ref{sec:ODEs}), metric graphs (Section \ref{sec:PDEs}) and embedded metric graphs (Section \ref{sec:embed_PDEs}). 

As a result, in this paper we deal with a wide range of models. Many of them offer different ways of space and time discretisation, they also vary in incorporating subsystems into the main structure. Even so, one can encounter clear mathematical bridges that fasten some parts of theories with others. 
This supports the idea that there are notable similarities between network models that should be explored further. Our reflections on this can be found in Section \ref{sec:interdep}, where we summarise discussions held among a wide range of experts in dynamical systems on networks, as part of the COST Action Mathematical models of interacting dynamics on networks.

To the best of our knowledge, this is the first paper that encompasses such a broad perspective on the world of dynamical systems on networks. We invite the reader on a journey where we avoid highways of general theories and choose country roads to look in detail into some very specific aspects of selected diverse topics. We hope that on this trip many network specialists will find the home base of their expertise. This may be a good point of departure. Depending on the reader’s interests, for some topics taking in the view from the window might be enough, but any time one decides on a stopover, there are references that direct from this small prospectus to more detailed mathematical studies of the topic. 

Finally, we hope that doing the entire tour brings us closer to the understanding of what are the major factors that should be taken into account when choosing different mathematical apparatuses. The wider perspective should highlight gains and losses in the choice of different mathematical settings. Finally, including such different methods in one study indicates that our world of networked dynamical systems is indeed a global village and we can draw inspirations and exchange concepts from one another. Buon viaggio!

\section{Problem Descriptions}\label{sec:problems}
Let us commence by presenting examples of applied problems that involve complex phenomena and can be tackled using the theory presented in the further sections. For each problem, we propose one specific model that offers potential solutions, providing at least partial insights into the challenges these problems present. 
\subsection{Ecology Inspiring Economy}

One of the natural examples of networks in life sciences is food webs. They describe the physical foundation of an ecosystem as mass flows between functional groups of species and their exchange with the environment. They can be regarded as akin to economic networks, as both consist of vertices processing and exchanging matter with each other~\cite{hannon1973structure,LEONTIEF1991181}. This enables us to compare properties of these two types of networks and draw inspiration from nature for man-made systems. For example, comparing cycling in these systems is motivated by the goal of a circular economy that aims to make production processes more sustainable~\cite{ghisellini2016review,webster2017circular}. Such attempts at mimicking food webs were postulated in the context of thermodynamic power cycles~\cite{Layton2012}, industrial parks~\cite{Layton2016}, recycling networks~\cite{schwarz1997implementing,Layton2016Bras}, and the whole economy in view of the limits to growth~\cite{FathLimits}. What can we learn from natural network processes? See Section~\ref{eco_net}.

\subsection{The Riddle of Language Acquisition}

Despite recent advances in neural machine translation and large language models facilitating successful communication across and text rendition into other languages with the help of an app, foreign language learning is still a valid goal for large numbers of students. In today’s increasingly interconnected world, more and more are deciding to spend at least part of their academic journey at a foreign university. Alongside being attracted by fresh course offerings, the opportunity to travel, experience a new culture and to gain new friends, as well as scoring points for their future résumés, one of the chief motivations behind study abroad is advancing one’s skills in the language of the target country. However, during their stay abroad not all sojourners make visible linguistic progress, and considerable variation has been evidenced among those that do. It is commonsensical to expect that a crucial role in second language development will be played by students’ interaction networks, but how exactly does participants’ position/centrality affect their progress across different linguistic competencies \cite{Paradowski1,Paradowski2}, what is the impact of the subcommunities/clusters that regularly form in peer cohorts \cite{Cierpich}, and how do students’ language learning trajectories and social behaviours pattern over time \cite{Brodka}? We address these questions in Section \ref{language}.

\subsection{Missing Data in Financial Stability Modelling}
\label{subsec:problem_missing_financial_data}

The Bank for International Settlements and the related database (owned by 60 central banks and maintained by 48 reporting countries) supports central banks in their pursuit of monetary and financial stability. In 2017 it represented 94\% of all cross-border claims of banks and covered around 95\% of the world’s Gross Domestic Product.
The Bank for International Settlements database is incomplete, since only some estimates can be obtained from the non-reporting countries.
Can the available dataset be used for estimating the effects of financial contagions, i.e.\ a cascade failure? Will the result be reliable? Or, wording it differently, do the missing links play a vital role in the contagions dynamic \cite{Cinelli}? Above questions are examined in Section \ref{bank}.

\subsection{Subscribing to a Pension Plan}
\label{subsec:problem_pension_plan}
For the sustainability of the future pensions, pension funds need to propose investments based on stock markets. 
Should a potential subscriber of a specific pension fund have 
clear preferences regarding the reliability of the pension fund when examining such investments? Are the actual investments overall a good choice for the stability of the pension fund system?
The present work contributes to showing how complex network analysis helps to understand the risk involving the whole set of Italian pension funds. In our analysis each  pension fund is a network, and its set of investment is represented through a list of  declared benchmarks. A link joins two nodes of the network anytime they declare a common benchmark. 
In Section \ref{pension_fund}, it is shown how clustering algorithms can help in understanding addressed problem. The results show quite a remarkable overlap among pension funds’ investments. These findings prove that eventual fluctuations of even a few benchmarks may cause serious consequences in the financial wealth of pension funds's investments, so the system is very fragile \cite{GRAMDJBRLevantesi}.
Moreover, information on pension funds's investments based on the stock markets provides very limited indications for the selection of the best pension fund to subscribe to.

\subsection{Network Synchronisation}
\label{subsec:network_sync}
 An important problem in the study of dynamical systems over networks concerns the emergence of coherent behaviour in which the elements of the system follow some dynamical pattern, i.e.\ are synchronised. In the presence of external disturbances of the system, one may want to measure or reduce the impact of disturbances on the synchronisation; the corresponding problem is called almost synchronisation. Specific applications include flights of satellite formations, distributed computing, robotics, surveillance and reconnaissance systems, electric power systems, cooperative attacks involving multiple missiles, intelligent transportation systems, and neural networks. To quantitatively analyse synchronisation properties,in Section \ref{synchron} 
 the following question is addressed: if one is able to disturb/bolster merely one agent in order to maximally disturb/bolster the entire system, which agent should one choose? This question is obviously related to the almost synchronisation property, and in this way one can find an ordering of agents according to their sensitivity to the external disturbances.

 \subsection{Network identification with few measurements}
\label{subsec:network_ident_prob}

 Network identification roughly consists in revealing the structure of a graph from data generated by agents located at the vertices of the graph and connected to other agents according to the network topology. This problem arises in many contexts, such as neuroscience (e.g. capturing the connections between cerebral regions from brain imaging data), genetics (e.g. inferring gene regulatory networks from gene expression data), social sciences (e.g. measuring influences between individuals on social media), and finance (e.g. detecting interconnectedness between financial institutions), to list a few. Those situations usually involve large-scale networks, where measuring all the vertices is out of reach. Moreover, since measurements can possibly be costly or time-consuming, there is a need for methods that only require \emph{local} measurements at a few vertices, but which are still capable of revealing \emph{global} topological properties in large networks. Specific problems include, for instance, inferring the average number of connections between the agents from sparse measurements, detecting a change in the network topology from a remote measurement, and assessing the mutual influence between agents (e.g., community detection). This problem will be tackled in Section \ref{subsec:spectral_net_ident} in the context of spectral network identification. It is also related to the problems presented in Sections \ref{subsec:problem_missing_financial_data} and \ref{subsec:problem_pension_plan}.

\subsection{Tracing Virus Variants in a Fragmented Environment}\label{subsec:virus}
In 2019, due to the outbreak of the SARS-CoV-2 pandemic, the problem of monitoring the spread of infectious diseases changed into a global challenge. In the first years of the pandemic, the major question was how to properly estimate the real number of colonised and infected patients in order to first estimate the basic reproduction number, known as the $\mathcal{R}_0$ parameter and, consequently, to navigate wisely in policy responses. Since December 2020, when the first COVID-19 vaccine was released, more and more attention has been focused on the problem of tracing new variants, in particular those 
that exhibit reduced effectiveness to the vaccine. Modelling the prevalence of patients colonised by the particular variant of a SARS-CoV-2 virus in the local environment allows for either the production of more tailor-made vaccines, similarly to flu vaccines, or at least for the wider availability thereof among different vaccines that should be more efficient for particular groups of variants. One possible form of such model is presented in Section \ref{subsec:virus2}.

\subsection{Optimizing Road Traffic}
Urban road transport poses significant challenges to civilization and economic activity, impacting quality of life and productivity. Addressing these issues requires interdisciplinary approaches and sustainable urban planning strategies to alleviate traffic-related negative effects. One such approach involves optimizing traffic signal settings for given traffic conditions, but this problem has been proven to be NP-hard even for relatively simple traffic models \cite{chen2007algorithms}. Also, even evaluating the quality of different traffic signal settings can be time-consuming, especially on a large scale. 
Answering a need to develop new traffic modelling and optimization methods the new approach based on both fluid dynamics' model and AI tools is proposed in Section \ref{subsec:traffic}.

\section{Preliminary Network Description}\label{sec:prelimi}

\subsection{Combinatorial Digraphs}\label{sec:combi_net}

The basic object of our consideration is a \emph{digraph} (also known as a \emph{directed graph}) $G=(V,E,\Phi^{\pm},W)$ with sets of \emph{vertices} $V=\{v_i\}_{i\in I}$, $I=\{1,\ldots,n\}$ and \emph{edges} $E=\{e_j\}_{j\in J}\subset V\times V$, $J=\{1,\ldots,m\}$. The graph structure is encoded by \emph{in- and out-incidence matrices} $\Phi^{\pm}=(\phi^{\pm}_{i,j})_{i\in I,j\in J}$ defined as
\begin{equation}\label{eq:inci_mtx}
\phi^{+}_{ij}=\left\{\begin{array}{ll}1 &\text{if}\quad \stackrel{e_j}{\rightarrow}v_i,\\
0 &\text{otherwise},\end{array}\right.
\qquad \phi^{-}_{ij}=\left\{\begin{array}{ll}1 &\text{if}\quad v_i\stackrel{e_j}{\rightarrow},\\
0 &\text{otherwise},\end{array}\right.
\end{equation}
and a \emph{weight matrix} $W=\textrm{diag}(w_j)_{j\in J}$, $w_j>0$, $j\in J$. In \eqref{eq:inci_mtx} we apply the notation $\stackrel{e_j}{\rightarrow}v_i$ (resp. $v_i\stackrel{e_j}{\rightarrow}$) for edge $e_j=(\cdot,v_i)$ (resp. $e_j=(v_i,\cdot)$) that ends (resp. starts) in a vertex $v_i$.

If all weights are equal, we call $G$ an \emph{unweighted digraph}. By an \emph{(undirected) graph} we understand a digraph having the property that for any $(v_i,v_j)\in E$, $(v_j,v_i)\in E$ and $W(v_i,v_j)=W(v_j,v_i)$. We also denote by $E_{v_i}^{\pm}$ the sets 
\begin{equation*}
E_{v_i}^+=\{e_j\in E:\,\, \stackrel{e_j}{\rightarrow}v_i\},\qquad E_{v_i}^-=\{e_j\in E:\,\, v_i\stackrel{e_j}{\rightarrow}\},
\end{equation*}
and their weights' counterparts $W_{v_i}^{\pm}$ 
\begin{equation}\label{eq:weights_vi_neigh}
W_{v_i}^{\pm}=\{w_j:\,\, e_j\in E_{v_i}^{\pm}\},\qquad W_{v_i}=W_{v_i}^++W_{v_i}^-.
\end{equation}

If $e_k=(v_i,v_j)\in E$, then $v_i$ is the \emph{tail} and $v_j$ is the \emph{head} of an edge. We say that a vertex $v_i\in V$ is a \emph{source} if $\sum_{j\in J}\Phi^+_{ij}=0$, and a \emph{sink} if $\sum_{j\in J}\Phi^-_{ij}=0$. Furthermore, $e_j$ is a \emph{loop} if there exists $v_i\in V$ being both its head and tail. By a \emph{multiple edge} in the digraph we understand at least two edges all having the head in $v_i\in V$ and the tail in $v_j\in V$.  By the $l$-length \emph{path} in the graph we understand a sequence of edges $e_1,e_2,\ldots,e_l$, $e_i\in E$, $i=1,\ldots,l$; such that for each $i=1,\ldots,l$, there exists $k_i\in I$ such that $\phi_{k_i i}^+=\phi_{k_i i+1}^-$. We say that a digraph is connected if for any two vertices $v_i,v_k\in V$ there exists a path such that its first edge has the tail in $v_i$ and the last edge has the head in $v_k$.

In this paper we restrict our consideration to connected digraphs having a finite number of vertices and edges, and no multiple edges (loops are allowed).

When considering the dynamics in the vertices of digraphs, it is convenient to define operators between vertices. By \emph{weighted in- and out-adjacency matrices}, $\mathcal{A}_w^{\pm}=(a_{ij}^{w\pm})_{i,j\in I}$, $(\mathcal{A}_w^{+})^T=\mathcal{A}_w^{-}$, we understand 
\begin{equation*}
a_{ij}^{w+}=a_{ji}^{w-}=\left\{\begin{array}{ll}w_{k} &\text{if}\quad \exists_{e_k\in E}\,\, v_j\stackrel{e_k}{\rightarrow}v_i\\
0 &\text{otherwise}.\end{array}\right.
\end{equation*}
If we replace $a_{ij}^{w\pm}\neq 0$ with $a_{ij}^{\pm}=1$, we say that $\mathcal{A}^{\pm}=(a_{ij}^{\pm})_{i,j\in I}$ is an \emph{(unweighted) in- and out- adjacency matrix}.

To give the number of edges entering or going out from the vertex and their cumulative weight we define respectively \emph{(unweighted) in- and out-degree matrices} $\mathcal{D}^{\pm}=\text{diag}(\deg^{\pm}(v_i))_{i\in I}$ and \emph{weighted in- and out-degree matrices} $\mathcal{D}_w^{\pm}=\text{diag}(\deg_w^{\pm}(v_i))_{i\in I}$ by
\begin{equation}\label{eq:deg}
\text{deg}^{\pm}(v_i)=\sum_{j\in J}\phi_{ij}^{\pm},\qquad\text{deg}_w^{\pm}(v_i)=\sum_{j\in J}\phi_{ij}^{\pm}w_j.
\end{equation}
 An \emph{undirected unweighted adjacency matrix} $\mathcal{A}=(a_{ij})_{i,j\in I}$ is defined as
 \begin{equation}\label{eq:unwundiradjacency}
     \mathcal{A}=\mathcal{A}^++\mathcal{A}^-,
 \end{equation}
and an \emph{unweighted degree matrix} $\mathcal{D}=\text{diag}(\deg(v_i))_{i\in I}$ as
 \begin{equation}\label{eq:unwdegree}
 \mathcal{D}=\mathcal{D}^++\mathcal{D}^-,
\end{equation}
while an \emph{undirected weighted adjacency matrix} $\mathcal{A}_w=(a_{ij}^{w})_{i,j\in I}$ by
 \begin{equation}\label{eq:wundiradjacency}
 \mathcal{A}_w=\mathcal{A}_w^++\mathcal{A}_w^-.
\end{equation}

Let us also mention two other basic operators that are used to characterise the dynamics associated with vertices, namely the advection and Laplacian matrix. In the case of directed graphs, these can be defined in various ways and consequently one can find many names for the same objects. In this paper we mostly follow the notation from \cite{Mugnolo:14}. Hence, \emph{weighted in- and out-advection matrices} are $\mathcal{N}_{w}^{\pm}=(N_{w,ij}^{\pm})_{i,j\in I}$ such that 
\begin{equation}\label{eq:adv}
\mathcal{N}_w^+=\mathcal{D}_w^--\mathcal{A}_w^+,\qquad 
\mathcal{N}_w^-=\mathcal{D}_w^+-\mathcal{A}_w^-.
\end{equation}
Whereas by a \emph{weighted in- and out-degree Laplacian matrix} $\mathcal{L}_w^{\pm}=(L_{w,ij}^{\pm})_{i,j\in I}$ (also known as an \emph{incoming and outgoing Kirchhoff matrix} and denoted by $\mathcal{K}_w^{\pm}=(K_{w,ij}^{\pm})_{i,j\in I}$) we understand
\begin{equation}\label{eq:laplac}
\mathcal{L}_w^{+}=\mathcal{K}_w^+=\mathcal{D}_w^+-\mathcal{A}_w^+,\qquad \mathcal{L}_w^{-}=\mathcal{K}_w^-=\mathcal{D}_w^--\mathcal{A}_w^-.
\end{equation}
 Finally, by $\mathcal{L}_w^B=(L_{w,ij}^B)_{i,j\in I}$ we denote a \emph{weighted Laplacian-Beltrami matrix}, namely
\begin{eqnarray}\label{eq:laplac-beltr}
\mathcal{L}_w^B=\mathcal{L}_w^++\mathcal{L}_w^-.
\end{eqnarray}
To obtain unweighted counterparts of advection $\mathcal{N}^{\pm}=(N_{ij}^{\pm})_{i,j\in I}$, Laplacian $\mathcal{L}^{\pm}=(L_{ij}^{\pm})_{i,j\in I}$ (known as Kirchhoff $\mathcal{K}^{\pm}=(K_{ij}^{\pm})_{i,j\in I}$) and Laplacian-Beltrami matrices, we consider unweighted degrees $\mathcal{D}^{\pm}$ and unweighted adjacency matrices $\mathcal{A}^{\pm}$ in formulae \eqref{eq:adv}, \eqref{eq:laplac}, as well as unweighted Laplacians in \eqref{eq:laplac-beltr}.

Finally, when considering dynamics defined on digraphs' edges, it is convenient to also define operators between the edges. A \emph{weighted in-adjacency matrix of a line graph of $G$} we call a matrix $\mathcal{B}_w^{\pm}=(b_{w,ij}^{\pm})_{i,j\in J}$ such that 
\begin{eqnarray*}
b_{w,ij}^+&=&\left\{\begin{array}{ll}w_j &\text{if}\quad \exists_{v_k\in V}\,\, \stackrel{e_j}{\rightarrow}v_k\stackrel{e_i}{\rightarrow}\\
0 &\text{otherwise},\end{array}\right.\\
b_{w,ij}^-&=&\left\{\begin{array}{ll}w_i &\text{if}\quad \exists_{v_k\in V}\,\, \stackrel{e_j}{\rightarrow}v_k\stackrel{e_i}{\rightarrow}\\
0 &\text{otherwise}.\end{array}\right.
\end{eqnarray*}
In the case of an \emph{unweighted adjacency matrix} $\mathcal{B}=(b_{ij})_{i,j\in J}$ \emph{of a line graph of $G$}, we consider $b_{ij}=1$ when $b_{ij}^{\pm}\neq 0$, and $b_{ij}=b_{ij}^{\pm}$ otherwise.

\subsection{Metric Graphs}\label{sec:met_net}

The second major group of networks used in the modelling of dynamical systems are metric graphs. This concept is a generalisation of digraph that permits associating each edge with one-dimensional metric space and consequently defining differential operators on the network. 

Using mathematical formalism, let $\mathcal{G}=(G,d)$ be a \emph{metric graph} where $G=(V,E,\Phi^{\pm},W)$ is a digraph defined in Section \ref{sec:combi_net} and 
$d:E\rightarrow d(E)\subset \reflectbox{D}$, for $\reflectbox{D}$ being the set of all real intervals, is a mapping of the form $ e_j \mapsto [0,l_j]$, $l_j>0$ for any $j\in J$. Note that $E=\{e_j\}_{j\in J}\subset V\times V$ is a set of edges of the digraph and $J=\{1,\ldots,m\}$. 

In the metric graph, the direction of an edge $e_j=(v_i,v_k)\in E$, $i,k\in I$, can be associated with the parametrisation of $d(e_j)$. Hence, by abuse of notation, we denote the tail (resp. head) of an edge $e_j$ by $e_j(0)=v_i$ (resp. $e_j(l_j)=v_k$). 

In metric graph models, the state space consists of intervals that are joined together by vertices in which they start and end. 
Finally, we want to associate the network with the space it is embedded into. We say that a digraph $G=(V,E,\Phi^{\pm},W)$ is a \emph{planar digraph} if it can be drawn on the plane in such a way that each edge $e_j$ is of the length $l_j$ consistent with its weight $w_j>0$, $j\in J$, and edges intersect only at their endpoints. For every such network, we can choose its planar representation by defining a mapping 
$P: V \rightarrow \mathbb{R}^2$ such that if the edge $e_j$ connects the vertices $v_i$ and $v_k$, then the Euclidean distance between $P(v_i)$ and $P(v_k)$ coincides with the length $l_j$ of $e_j$. We say that $G_P=(V,E,\Phi^{\pm},W, P)$ is a \emph{planar embedding of a digraph} $G=(V,E,\Phi^{\pm},W)$, while $\mathcal{G}_P=(G_P,d)$ is a \emph{planar embedding of a metric graph}. \newline

Finally, in the whole article we use the notation \begin{equation}\label{eq:Id&1}
I_n\in \mathbb{R}^{n\times n},\quad \mathbbm{1}_n=[1,\ldots,1]^\top\in \mathbb{R}^n,
\end{equation}
for the identity matrix and the vector that consists of all entries equal to one, respectively.

\subsection{Erdős–Rényi, Watts-Strogatz, and Barabási–Albert Networks}

A random network is a network where the edges are a subset sampled randomly from the set of all the possible edges. The independent variables of this model are the number of nodes and the number of edges. Any network with the same number of edges is equally likely to occur. 
The paper {\it Random networks} authored by 
Paul Erdős and Alfréd Rényi in 1959 fixed a milestone in the field \cite{erdHos1960evolution}. 
The random network model uses probabilistic methods for proving the existence of networks satisfying various properties, and it offers a framework for a rigorous definition of what it means for a property to hold for almost all graphs. 

Questions were addressed which are of particular interest in many fields spanning from percolation theory to modern telecommunication networks, such as the presence of a connected component
gathering the vast majority of nodes.
 Contemporaneously with and independently of Erdős and Rényi, Edgar Gilbert introduced the random network model formalized through the number of nodes and the probability of any edge to be present, instead of using the total number of edges.
 
This line of research attracted much interest in the topic. However, the model did not fit datasets, which were showing a diameter of the network relatively small if compared to the number of nodes. So-called 
{\it small world} networks became popular with the experiment of sociologist
Stanley Milgram, who in (1967) \cite{milgram1967small}
was looking for the number of intermediaries that are needed to connect two persons in the U.S. 
The experiment gave rise to the famous statement that any two persons in the U.S. have about  six degrees of separation—meaning that in principle any person can contact any other through at most six intermediaries \cite{travers1977experimental}. 
In 1998, Duncan J. Watts and Steven Strogatz elaborated on a network model 
which shows small-world properties, including short average path lengths and high clustering.
The model depends on a parameter and interpolates between a lattice and a random network.

The interest in network topology has been driven by datasets most when
{\it scale free} networks became popular. A scale-free network is a network where the empirical distribution of the node degrees can be modeled through a power law function. In 2002 Albert-L\'aszl\'o Barab\'asi and R\'eka Albert proposed a model which quickly became very famous \cite{albert2002statistical}. The model builds up a scale-free network through a mechanism of preferential attachment: starting from an initial “seed” network composed by a few nodes and links, new nodes and edges are added. The probability of an edge linking the new node to a specific other node already present in the network depends on the node degree of the latter. Scale-free networks have the property of possessing a few nodes with high degree (hubs) and many nodes with low degree. 
Datasets having a power law distribution include social networks, many kinds of computer networks, including the Internet and the webgraph of the World Wide Web, some financial networks, protein–protein interaction networks, and semantic networks.
Further extensions have considered directed networks, separating analyses of the in-degree from the out-degree.

\section{Combinatorial Digraphs}\label{sec:combinat}

In this we apply section statistical tools to reveal information encoded in network structures. Graph measures are considered  that characterise certain properties of a digraph: node degree, closeness, betweenness, PageRank, reciprocity, or Finn cycling index, to mention a few. Another approach is to extend the deterministic notion of a digraph into the Erdős–Rényi random structure or Barabási–Albert scale-free networks.

\subsection{Networked Ecosystems and Economies}\label{eco_net}
Studying the flow and cycling of matter in networks is crucial for understanding the fundamental physical structure of both ecosystems and economies. Food webs describe how biomass moves between groups of species in ecosystems, arising mostly from feeding relationships. Similarly, input-output tables capture the flow of goods and services between industries and consumers in an economy. Current data-based models capture averaged snapshots of the dynamical systems in question. They are balanced - the sum of flows into equals the sum of flows out of each vertex. Such networks typically correspond to steady states of corresponding dynamical models.

In~\cite{MatIsk_cycling}, 169 weighted food webs and 155 economic networks are compared based on the fraction of total system throughflow that is cycled, known as the Finn Cycling Index (FCI).


In order to learn more about quantification of mass' cycles, let us consider both food webs and economic networks as weighted connected digraphs $G=(V,E,\Phi^{\pm},W)$, see Section \ref{sec:combi_net}, with weights $W$ denoting the flow of matter along edges. In this study the vertices are additionally characterised by external export from each vertex $o=(o_i)_{i\in I}$, where $o_i\geq 0$ and $o$ is non-trivial. 

The amount of matter leaving a node $v_i\in V$, denoted by $h=(h_{i})_{i\in I}$, is assumed to be positive and is given by
 \begin{equation*}
     h_{i}=\sum_{j=1}^{n} a_{ji}^{w+} + o_i>0.
 \end{equation*}
A transition matrix $C=(c_{ij})_{i,j\in I}$ describes the probability that a unit of mass moves between adjacent nodes, namely
 \begin{equation*}
 c_{ij}=\frac{a_{ij}^{w+}}{h_{j}}\quad \text{for any}\,\,i,j\in I.
 \label{eq:transition}
 \end{equation*}
Finally, in order to include indirect flows between nodes, one defines overall transition probability $U=(u_{ij})_{i,j\in I}$ as the power series
\begin{equation*}
U =  \sum_{q=0}^{\infty}C^q=(I_n-C)^{-1}.
\label{eq:matseries}
\end{equation*}
The matrix $U$ is well-defined since $C$ is sub-stochastic, as digraph $G$ is connected.

\emph{The Finn Cycling Index} (FCI) \cite{finn} of a vertex is the probability that flow passing through a
node returns (via direct or indirect flow) to it at some time point, formally:
\begin{equation}
\text{FCI}_i= \frac{u_{ii}-1}{u_{ii}},\quad \text{for any}\,\,i\in I.
\label{eq:singlenode}
\end{equation}
Note that $u_{ii}\leq 1$ from the definition, and therefore FCI is a measure that ranges from 0 to 1. A value of $\textrm{FCI}=0$ indicates that there are no directed cycles in the system, meaning that no flow originating from any node returns to the same node through any path. Conversely, a value of $\textrm{FCI}=\mathbbm{1}_n$ represents a system in which all flows eventually return to their starting node. 

The most basic form of cycling occurs when two nodes exchange mass reciprocally. Network \emph{reciprocity} $r$ measures the fraction of such overlapping bilateral flows among all flows:
\begin{equation}
r= \frac{\sum_{i,j\in I} \min \{u_{ij},u_{ji}\}}{\sum_{k,l\in I} {u_{kl}}}\in[0,1].
\label{eq:reciprocity}
\end{equation}
If all flows are perfectly reciprocated, then $r = 1$. If there are no nodes connected by flows in both directions, $r = 0$. 

It turns out that the FCI of food webs has a geometric mean of 5\%, with much of the cycling attributed to reciprocal flows and the recycling of dead organic matter by detritivores. In contrast, the global economy in 2011 had an FCI of 3.7\%. Furthermore, \cite{MatIsk_cycling} highlights that unweighted network measures used in the past, such as the largest eigenvalue of the adjacency matrix, do not correlate with the actual cycling in weighted networks.

Interestingly, both food webs and economic networks exhibit a strong correlation between FCI and reciprocity, defined as the fraction of flow that is immediately returned between two nodes. This suggests that promoting reciprocity and local collaboration between network components could be a simple strategy to enhance cycling without requiring global knowledge of the system structure. The study emphasizes the importance of relying on weighted network indicators to make sound inferences about real-world systems.

A related work~\cite{MatIsk_foodwebviz} developed an open-source Python package called foodwebviz for visualizing weighted food webs. The package offers five complementary methods: 1) a heat map of flows or diet proportions, 2) an interactive graph for tracing matter flow, 3) an intuitive animation of particles moving between nodes, 4) a summary bar plot of trophic level exchanges, and 5) a heat map of trophic level flows. These tools facilitate accompanying food web publications with clear, aesthetically appealing visualizations that can engage the broader public and be incorporated into education.

Together, these studies contribute to our understanding of mass cycling in complex networks and provide practical tools for their analysis and communication. They highlight the insights gained from weighted network approaches and the potential for local strategies to enhance sustainability in both ecosystems and economies.

\subsection{Peer interactions and Second Language Acquisition during Study Abroad}\label{language}

In today’s increasingly interconnected world, more and more students are deciding to spend at least part of their academic journey at a foreign university. In this section we study factors conductive to second language (L2) development as well as possible barriers. We are also interested in the temporal evolution of students’ language learning trajectories and social behaviour patterns.

A successful line of inquiry permitting answers to these questions is offered by social network analysis 
\cite{Dewey,Hasegawa,Kennedy,Mitchell1,Mitchell2,Paradowski1,Paradowski2,Cierpich}. Particularly fruitful have been computational analyses that go beyond investigations of egocentric networks and instead attempt to reconstruct possibly complete learner graphs. 

From a mathematical perspective, one can define an undirected graph $G=(V,E,\Phi^{\pm},W)$ where vertices denote students while edges indicate social relations between them. A student's communication with a larger number of peers, and in turn those peers' mutual interactions, decreases the distance between the focal student and others and thus increases their centrality. Consequently, if by $d(v_i,v_j)$ we denote the weight of the shortest path between vertices  
$v_i,v_j\in V$, then by closeness centrality of vertex $v_i\in V$, $C_C(v_i)$ we understand
$$C_C(v_i)=\sum_{j\in I}\frac{1}{d(v_i,v_j)}.
$$
Another path-based graph measure is betweenness centrality. We denote by $\sigma_{ij}$ the number of paths between $v_i$ and $v_j$ with lengths equal to the length of a shortest path $d(v_i,v_j)$. Analogously, $\sigma_{ij}(v_k)$ is the number of paths from $\sigma_{ij}$ that pass through $v_k$. Then betweenness centrality reads
$$C_{B}(v_i)=\sum_{j\in I\setminus \{i\}}\sum_{k\in I\setminus \{i\}}\frac{\sigma_{jk}(v_i)}{\sigma_{jk}}.$$
A third standard graph measure that has found application in the context of language learning via peer interaction is the degree centrality of a vertex 
$$C_D(v_i)=\deg(v_i),$$ where $\deg(v_i)$ is defined in \eqref{eq:unwdegree}.

An analysis of a cohort of 39 Erasmus+ students at a German university employed social network analysis 
with cluster detection algorithms to examine language learning patterns. This approach proved effective for understanding how different types of social connections influence language acquisition.

The study revealed several key findings about social interactions and language learning. First, students who formed subcommunities with co-nationals and other speakers of the same native language showed reduced language acquisition. Second, high \emph{closeness centrality}—the degree to which a student was “close” to the others—also had a negative impact on language progress. This suggests that students benefit more from maintaining fewer, but deeper and longer interactions that demand use of richer lexical resources and more advanced syntactic constructions, rather than engaging in many necessarily briefer and more superficial conversations over the same amount of time available \cite{Cierpich}. Additionally, the analysis unearthed heterophily in second-language (L2) German proficiency among close contacts, revealing the understandably negligent impact of this criterion in choosing friends. The study also highlighted the importance of actively using the target language versus merely being exposed to it passively.

Additional research in two iterations of a large-scale course of the Polish language and culture in Warsaw revealed how the centrality metrics best predicting measurable progress are \emph{closeness} and \emph{degree}, with outdegree—see~\eqref{eq:deg}—and \emph{betweenness} additionally explaining subjectively perceived headway in vocabulary and pronunciation, the latter negatively \cite{Paradowski1,Paradowski2}. Degree centrality reflects the number of a focal student’s alters – persons they have been able to build and stay in a relation with. A sub-dimension of degree representing the number of alters the student reported talking to is known as out-degree, while in-degree reflects how many others indicated that student as their interlocutor. In the context of language acquisition, these two metrics might be treated as imperfect proxies of the number of people exposed to a given student’s output, and the number of those from whom the student receives input, respectively. Betweenness in turn stands for the percentage of all the shortest paths in the network that pass through the student. Its negative impact in peer learner networks on the acquisition of second language vocabulary and pronunciation may be due to its reflecting higher exposure to non-native models, possibly resulting from identification with co-nationals. The study also showed background-language-based homophily among students frequently interacting with one another, with clusters in the stochastic blockmodels often aligning with learners’ shared languages (Figure \ref{fig:blockmodel}).

\begin{figure}
    \centering
    \includegraphics[width=.5\linewidth]{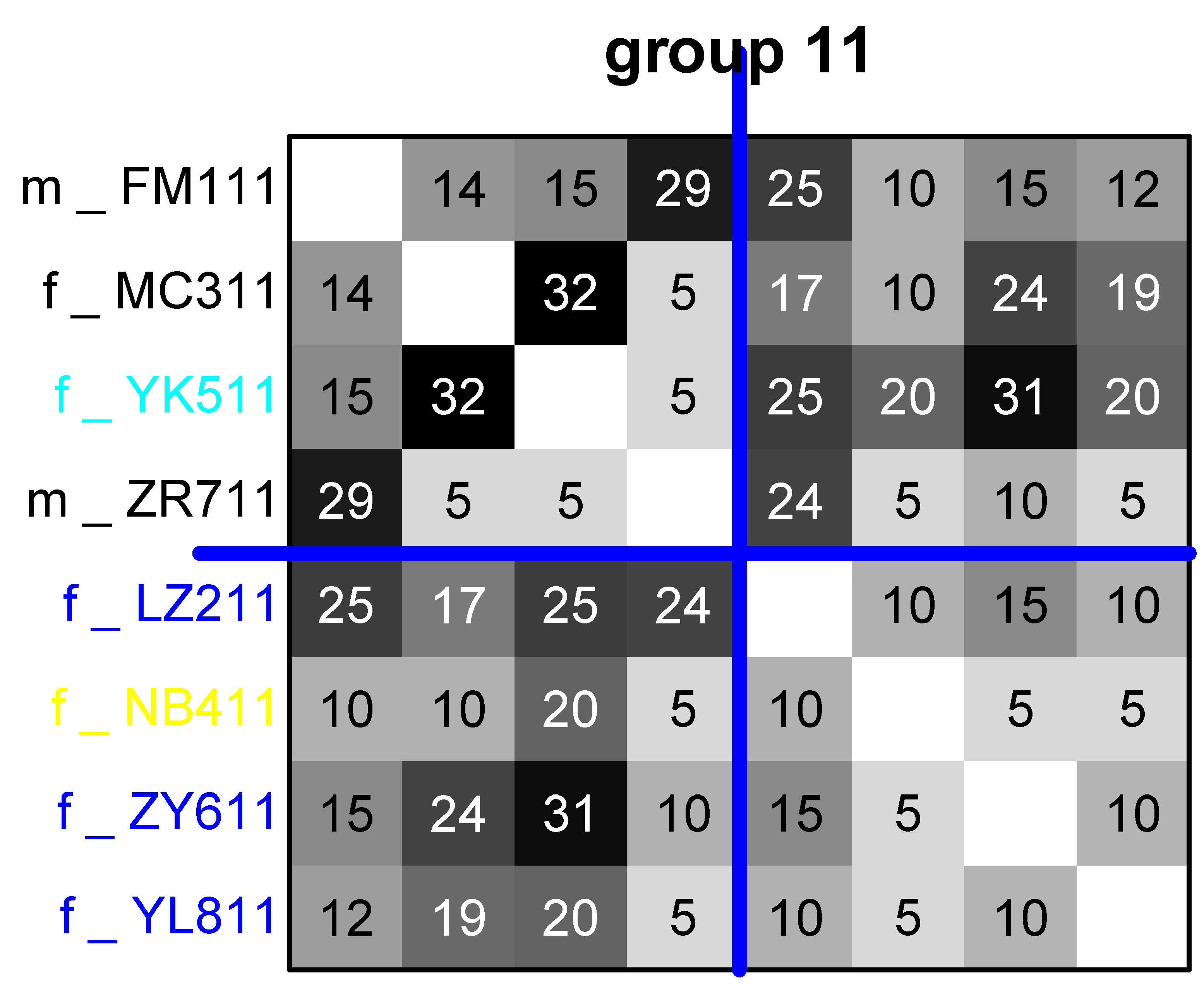}
    \caption{Class of students enrolled in an intensive course of the Polish language and culture showing clear partitioning along two shared background languages, German and Russian.}
    \label{fig:blockmodel}
\end{figure}

The static data have already yielded some knowledge, but what if we are interested in whether observable changes can be noticed in the patterns of students’ social interaction over time? Questions concerning the changing dynamics of second language development and peer communication can be answered employing longitudinal network analyses with several time points/snapshots, a new social network analysis' paradigm in quantitative L2 research. 

One recent project followed a whole group of 41 U.S. students enrolled in an intensive 3-month Arabic program in Amman, measuring their social interaction and progress after every four weeks. The findings revealed relative stability in terms of students’ positions in the network, with a stable tendency to form cliques with peers of the same gender (Figure \ref{fig:Amman}), and the gender homophily strengthening with time \cite{Brodka}. The results also showed that two out of the four significant predictors of objective progress were connected with social interaction: \emph{indegree} and perception of group integration, that presojourn proficiency in Arabic negatively influenced initiating interactions in this language, and that female students were spending considerably more time with their alma mater classmates.

\begin{figure}
    \centering
    \includegraphics[width=1\linewidth]{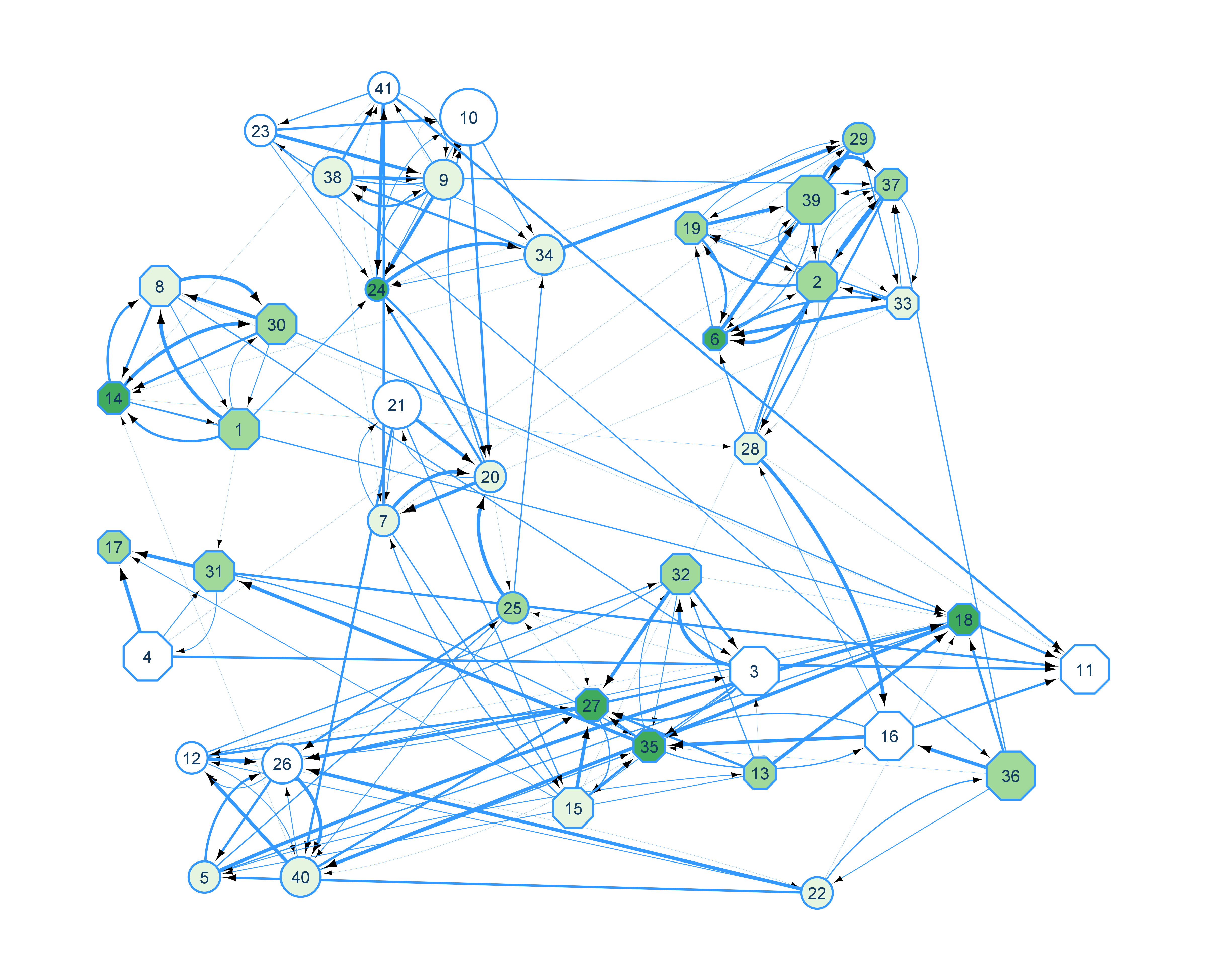}
    \caption{Students’ interactions during the middle month of their study-abroad sojourn in Jordan. Circles denote female, octagons male participants. Node color intensity reflects presojourn Oral Proficiency Interview score in Arabic; node size objectively measured postsojourn progress. Arrows show each student’s top five declared contacts/interlocutors (upper limit set for transparency), with edge thickness reflecting reported frequency/intensity of communication.}
    \label{fig:Amman}
\end{figure}

Network science has also led to epistemological contributions in other fields of linguistics \cite{Jarynowski}, shedding light on phenomena such as the social diffusion of innovation, or the relationships between constructs in educational psychology. For instance, investigations of the spread of neological tags online \cite{Jonak} showed the research benefits of gaining authorized backend access to an entire microblogging site, thus permitting insight into the degree of the “saturation” of the system rather than merely raw numbers of occurrences of expressions of interest, revealing how resistance against the adoption of novelty concentrates around relatively low values, that the vast majority of the users are moderately innovative, and that the spread of the most popular tags may be taking place during most users’ offline time. Language education research has in turn benefited from psychological network analysis \cite{Epskamp}, a variable-centered relation-intensive approach helpful in visualizing relationships between constructs and estimating the relative importance of factors in complex networks of associations. This type of network analysis has so far received relatively little attention from linguists \cite{Freeborn}, but has already been used e.g. to provide insight into the non-trivial relationship between grit and its predictors in online language learning \cite{Jelinska}.


\subsection{Network Identification from the  Bank for International Settlements Database}\label{bank}

The process of decision-making is often based on information available through databases that 
in many cases are incomplete. The
results in \cite{Cinelli}, aim at understanding how bad is the missing information in the Bank for International Settlements  database arising due to non-reporting countries.

Data are used to build up a digraph $G$, where the countries are the nodes, and the weights of the directed links report positive cross-border exposure. A contagion dynamic is set up considering two possible states of a node $v_j$ connected to a node $v_i$: either Credit (node $v_j$ has a positive credit to node $v_i$) or Debt (node $v_j$ has a debt to node $v_i$) and two possible events on the node $v_i$: either Loss (node $v_i$ loses value) or Gain (node $v_i$ gains value).
First, simulations are performed on the  network as it is from the dataset, and then on simulated \emph{random networks} (Erdős–Rényi) and \emph{scale free} networks (Barabási–Albert model).
As the second step, the total number of links in each network is increased in two different ways: either at random, or by adding more links among less-connected nodes ("organized periphery").

The performed analysis clearly shows that incomplete data leads to misleading information about 
the robustness of the network.
An increase in links supposed to be missing increases the magnitude of the contagion; the effect is at its worst in the case of organized periphery. The main conclusion is that the usage of the database as it is may lead 
to potentially biased actions which could not achieve the goal of stabilizing the system.

\subsection{Would you Subscribe to my Pension Plan?}\label{pension_fund}

Pension funds account for quite a remarkable amount of
the Gross Domestic Product, which in the OECD area averages 50.7\%. Pension funds are able to pay the pensions through investments. We focus the analysis on similarity among the investments, according to  the selection of the investments in risky assets from
the financial markets and represented through the declared benchmark. A "declared benchmark" in Italian pension funds is a performance standard or index used to compare the fund’s returns. It serves as a reference point to assess the fund's performance relative to the market or specific investment objectives. This benchmark helps ensure transparency, accountability, and provides a clear measure of whether the fund is meeting its goals.
Although economic relevance is clear, the literature on pension funds is quite limited. Our aim is to understand similarities among the pension funds through tools and 
measures proper for a complex networks approach.

The starting point is building an undirected graph $G$ in which nodes $v_i\in V$, $i\in I$ are the pension funds while edges' weights $w_j$, $j\in J$ measure the overlap among pension funds due to the overlap on the benchmarks, with $$W_V=\sum_{j\in J}\,w_j$$ being the sum of all weights in this network. For that, consider the database reporting the Italian pension funds, enumerated over $I$, and the declared benchmarks, enumerated over $L$. Data  has been provided by MEFOP (a society created by the Italian Ministry of Economics and Finance for the development of the market of pension funds) and refers to the year 2017. The set includes $61$ active sub-funds (belonging to $49$ funds) with their $72$ self-declared benchmarks. Data has been crosschecked through the Bloomberg database \cite{Bloomberg}.  
In detail, the data report information on {\it open pension funds} (that anybody can subscribe to).
Let $H=(h_{il})_{i\in I, l\in L}$ be a matrix reporting the pension funds in the rows and the benchmarks in the columns. $h_{ij}\geq 0$ is the percentage of the investment of the fund $v_i$ in the pension fund for a benchmark $j$.
Let $$\mathcal{A}_w=(a_{ij}^{w})_{i,j\in I}=H\cdot H^T$$ be a weighted undirected adjacency matrix, see \eqref{eq:wundiradjacency}, of a digraph $G$.
Note that the object is well-defined since both rows and columns of $\mathcal{A}_w$ are the pension funds, and entry $a_{ij}^w$ measures the overlap among pension funds $v_i$ and $v_j$. We are going to explore similarities through the application of two clustering methods: the \emph{Louvain method} \cite{Louvain} and $k$\emph{-shells} \cite{k-shells}.

The Louvain method follows a bottom-up approach in which
the nodes forming subnetworks as close as possible to complete subnetworks are gathered together. More formally, let us consider a partition of a graph $G$ into subnetworks $\{G_l:\,\, l\in L\}$ and define a Dirac function $\delta:V\times V \rightarrow \{0,1\}$, such that $\delta(v_i,v_j)=1$ when $v_i,v_j$ are in the same subnetwork. The communities are the elements of the partition $\{G_l:\,\, l\in L\}$ which maximises
the {\it modularity} given by the formula
\begin{equation}\label{eq:modularity}
Q\left(\{G_l:\,\, l\in L\}\right)=\frac{1}{2W_V} \sum_{i,j\in I} \left[a_{ij}^w-\frac{W_{v_i}W_{v_j}}{2W_V}\right]\delta(v_i, v_j),
\end{equation}
where $W_{v_i},W_{v_j}$ are the sums of the weights of the edges attached to nodes $v_i$ and $v_j$ respectively; see \eqref{eq:weights_vi_neigh} for definition.
The formula (\ref{eq:modularity}) represents the  aggregation because it measures to which extent a network is divided into communities compared to a random baseline. In fact, it measures the difference between the actual density of edges within communities (matrix $a_{ij}^w$) and the expected density if edges were placed at random (value $\frac{W_{v_i}W_{v_j}}{2W_V}$), while preserving the node degree distribution. The multiplication by $\delta(v_i, v_j)$ ensures that only edges within the same community $v_i=v_j$  contribute to the modularity score. The  factor $\frac{1}{2W_V}$  normalizes the modularity values to fall within the range of $[-1,1]$.
This captures the idea of community aggregation because it evaluates how tightly connected the nodes within a community are compared to connections across different communities.
The Louvain algorithm proceeds in two phases. First, the Modularity Gain: nodes are moved between communities to maximize the local contribution to modularity. Second, the Community Aggregation: Communities are treated as supernodes, and the process is repeated iteratively.
This hierarchical approach refines the community structure and produces an efficient maximization of the  modularity.

A $k$-shell of order $n$, $n\geq k$ is defined as the set of nodes which have degree at least $n$ after all the nodes with degree at maximum $k-1$ have been removed. The procedure for the identification of the $k$-shell is iterative.
The higher $k$, the higher the connection among the nodes in the same shell.

The Louvain method revealed that the Italian open pension funds are quite strongly connected.
The Louvain method distributes them across all the communities; a subgroup is also in the same community of the {\it contractual} pension funds, which are intended to specific
categories of workers, that can be set up on the basis of collective agreements. The $k$-shell method confirms the results of the Louvain method to a large extent \cite{GRAMDJBRLevantesi}.

The common reference of Italian pension funds to very few benchmarks exposes the entire set of pension funds to be highly sensitive to market fluctuations,
so the system shows to be quite fragile.

\section{ODEs on Combinatorial Graphs}\label{sec:ODEs}

In this section, we consider dynamical systems over networks, also called networked dynamical systems. 
Let us consider a (large) number $n$ of agents in interaction. Each agent indexed by $i$ is described by its internal state $x_i$ which evolves in time both because of an internal dynamics and because of coupling with the other units. A quite general form to describe the evolution of the system is given by the set of equations
\begin{equation}
    \label{eq:gen_ODE}
    \frac{dx_i}{dt} = \overline{}{F}_i(x),
\end{equation}
where $ x = (x_1,\ldots,x_n) $. 
The set of agents and their evolution equations \eqref{eq:gen_ODE} can be seen as a digraph $G=(V,E,\phi^{\pm},W)$, see Subsection \ref{sec:combi_net}, and two vertices $ v_i $ and $ v_j $ are linked by a directed edge from $v_j$ to $v_i$ if the evolution equation of agent $i$ depends on the state $ x_j $ of agent $ j $.

In particular, we focus on dynamical systems over networks that consist of
identical agents (or units) interacting through a diffusive coupling:
    \begin{equation}
            \label{eq:DSN_states}
            \frac{dx_i}{dt}  =  F(x_i) + K(x_i) \sum_{k  =  1}^{n} a^+_{ik} H(Q(x_i) - Q(x_k))  \quad x_i \in \mathbb{R}^p,\,i\in I,
    \end{equation}
with the functions $F : \mathbb{R}^p \to \mathbb{R}^p$ such that $F(0)=0$, $K : \mathbb{R}^p \to \mathbb{R}^{p\times r}$, $H : \mathbb{R}^q \to \mathbb{R}^r$, and $Q : \mathbb{R}^p \to \mathbb{R}^q$ such that $Q(0)=0$. Coefficients $a_{ik}^+$, $i,k\in I$ are the entries of a matrix $\mathcal{A}^+$, namely the unweighted in-degree adjacency matrix of $G$ defined in Section \ref{sec:combi_net}.

\subsection{Network Synchronisation}\label{synchron}

When a large number of agents are coupled through a complex network of interactions, these interactions can lead to cooperative phenomena and emergent properties of the overall dynamical system. An agent can be a model of, for example a ground/underwater vehicle, an aircraft, a satellite, or a smart sensor with microprocessors, while interactions can typically be modelled as information exchange. Problems of this type include formation control, flocking, distributed estimation, and consensus, and appear in different disciplines, including biology, physics, robotics, and control theory and problems, see e.g.\ \cite{Barrat_zbMATH05377449}. 

Here we deal with the output synchronisation problem, which can be described as follows.
Let the states of dynamical systems over networks be described by \cref{eq:DSN_states} with the underlying network structure described by the graph $G $ and with the corresponding agent outputs given by 
\begin{equation}
    \label{eq:outputs}
    y_{i} = Q (x_i).
\end{equation}
The goal of the output synchronisation \cite{lewis2013cooperative} is to reach asymptotic agreement between all agents, i.e.\ 
\begin{equation}
    \label{eq:synchronization}
    \lim_{t \to \infty} (y_i(t) - y_j(t)) = 0, \quad \textrm{for all}\,\,i,j\in I. 
\end{equation}
One can look at this notion as a natural replacement of stability for dynamical systems on networks. Indeed, take 
$ F = 0$, $K = H = Q = 1$ in \cref{eq:DSN_states};
then, the dynamical systems over network corresponds to a so-called network of single-integrators (a network without internal dynamics)
\[ \frac{dx_i}{dt} = 0,\,\, i \in I, \] 
coupled by the linear feedback with outputs being the states of the agents. This system is invariant to translations, i.e.\ if we translate the initial conditions $ x_i(0) $ of all agents by the same amount, the differences between the states will not change. One can easily check this by noting that the vector $ \mathbbm{1}$, see \eqref{eq:Id&1}, is the (right) eigenvector of $\mathcal{L}^+$. The same holds for dynamical systems over networks of double-integrators (second-order control systems of the form $\frac{d^2x}{dt^2}x=u$, $y=x$) and other important systems. This implies that such systems are always unstable and asymptotic stability does not hold. 
As an example, let us consider a general linear dynamical systems over network with second order dynamics in vertices. Let 
\[ \frac{d^2x_i}{dt^2} = - \mathcal{A} \frac{dx_i}{dt} + \mathcal{B} u_i, \quad y_i = \begin{bmatrix} x_i \\ \frac{dx_i}{dt} \end{bmatrix}  \]
where $ (\mathcal{A},\mathcal{B}) $ is a stabilizable pair, and $ u_i $ is the coupling given by 
\begin{equation}
\label{eq:protocol}
    u_i = - \begin{bmatrix} C_1 & C_2 \end{bmatrix}  \sum_{k = 1}^n a^+_{ik} 
\left( \begin{bmatrix} x_k \\ \frac{dx_k}{dt} \end{bmatrix} - \begin{bmatrix} x_i \\ \frac{dx_i}{dt} \end{bmatrix}  \right) .
\end{equation}
Then the system matrix is given by
\[ \mathcal{C} =  \mathcal{I}_{2n} \otimes \begin{bmatrix} 0 & \mathcal{I}_n \\ 0 & - \mathcal{A} \end{bmatrix} - \mathcal{L}^+ \otimes \begin{bmatrix} 0 & 0 \\ - C_1 & - C_2 \end{bmatrix}.  \] 
It is easy to check that the kernel of $ \mathcal{C} $ contains vectors of the form $ \mathbbm{1} \otimes [v \; 0]^\top $ and the system is again invariant to translations. 

In the presence of exogenous disturbances $ w_i $ of the agents, one would like to make the norm of the mapping $ \bar{y} \mapsto y_i - y_j $ smaller than some given tolerance $ \gamma $, where $ \bar{y} = [\bar{y}_1 \, \ldots \, \bar{y}_n] $. Typical choices for norms are $ H_2 $ and $ H_\infty $, as both norms measure the impact of exogenous disturbances given in the form of a stationary stochastic process. To be more precise, the $ H_2 $ norm can be interpreted as the output variance for disturbances described by white noise, while the $ H_{\infty} $ norm measures the worst-case effect a disturbance can make on the system, see \cite{dullerud2013course}. 
In this setting we talk about $ H_2 / H_\infty $\emph{almost output synchronisation}. 

If the underlying graph is undirected and connected, then one can transform the system, without loss of information, to a system with a stable system matrix. Indeed, then $\mathcal{L}^+ = \frac{1}{2}\mathcal{L}$ is symmetric and $ \mathbbm{1} $ spans the kernel of $ \mathcal{L} $. Let the matrix $\mathcal{Y}$ be such that its columns span the subspace $ \{ \mathbbm{1}\}^\perp $ and such that $ \mathcal{Y}^\top \mathcal{Y} = \mathcal{I} $. If we define $ \mathcal{Z} = \mathcal{Y}\otimes \mathcal{I} $ then the substitution $ x = \mathcal{Z} x' $ defines a reduced system with a stable system matrix. From the solution of the reduced system one can recover the state of the original system up to translation invariance. 

In the case of a directed graph $ G $, the situation is more complicated. In general, the vector $ \mathbbm{1} $ is only going to be a \emph{left} eigenvector corresponding to the eigenvalue zero. For example, the in-degree Laplacian of the following simple digraph
\begin{center}
\begin{tikzpicture}[roundnode/.style={circle, draw=black!60, fill=black!5, very thick, minimum size=7mm}]

\node[roundnode](midcircle){1};
\node[roundnode] (leftcircle)  [left=of midcircle] {2};
\node[roundnode] (rightcircle) [right=of midcircle] {3};

\draw[->] (leftcircle.east) -- (midcircle.west);
\draw[->] (rightcircle.west) -- (midcircle.east);

\end{tikzpicture}
\end{center}%
has (algebraic and geometric) multiplicity 2. Hence, a reduction similar to the one described for undirected graphs is not possible as we would lose too much information. Moreover, the corresponding dynamical systems over network does not have the property of output synchronisation. This can be easily seen by investigating a network without internal dynamics. The solution of the corresponding ODE is given by 
$ x(t) = \mathrm{e}^{-\mathcal{L}^+ t} $  where 
\[ \mathcal{L}^+ = \begin{bmatrix} 2 & 0 & 0 \\ - 1 & 0 & 0 \\ - 1 & 0 & 0 \end{bmatrix} , \]
and one obtains $ |x_2(t) - x_1(t)| \to \infty $ as $t \to \infty$ if $ x_2(0) \ne 0 $. The issue with this system is that both “leader” vertices 2 and 3 are influencing the “follower” vertex 1. It is clear that such couplings always lead to pathological systems. So the relevant dynamical systems over networks are those where the graph $ G $ is a \emph{directed forest}, i.e.\ a union of directed trees (subgraphs of graph $ G $ in which there exists a unique path from the root vertex to each node in this subgraph). Obviously, this means that it is sufficient to analyse graphs which contain a spanning directed tree, and indeed this is a standard assumption in the literature.  

Now we will address the question posed in subsection \ref{subsec:network_sync} in the special case of a network of agents modelled as  
an undirected graph with agents given as double-integrators, i.e.\ their dynamics is modeled as $$\frac{d^2 \chi_k}{dt^2} = -T_s \frac{d\chi_k}{dt} + K_s u_k +\omega_k,$$ with the coupling $u_i$ given as in \eqref{eq:protocol}, 
and using the $H_\infty$- norm as the measure of external impact. It was shown in \cite{NAKIC20229110} that one can efficiently calculate the influence of each agent on the system if, e.g.,\ $C_2\leq C_1T_S$. In fact, one can infer that the largest effect on the system will have constant disturbances and that  the greatest impact will be on the zero frequency component of the output.
Similar analyses can be made for general systems of identical linear agents. 

In \cite{NAKIC20229110}, an experimental validation of the theoretical results was performed. The experiment involved four nano quadrotors. As these drones are equipped with internal controllers that stabilize them, it is realistic to model them as double-integrators. The system identification yields parameters that satisfy $C_2\leq C_1T_S$ and the plots of the spectrum of the agents’ distance from the equilibrium (consensus) manifold given in Figure \ref{fig:amplitude_spectra} provide clear evidence that constant inputs affect this system more than signals at other frequencies and that the greatest impact is on the zero frequency component of the output.
\begin{figure}
        \centering
        \includegraphics[width=0.8\linewidth]{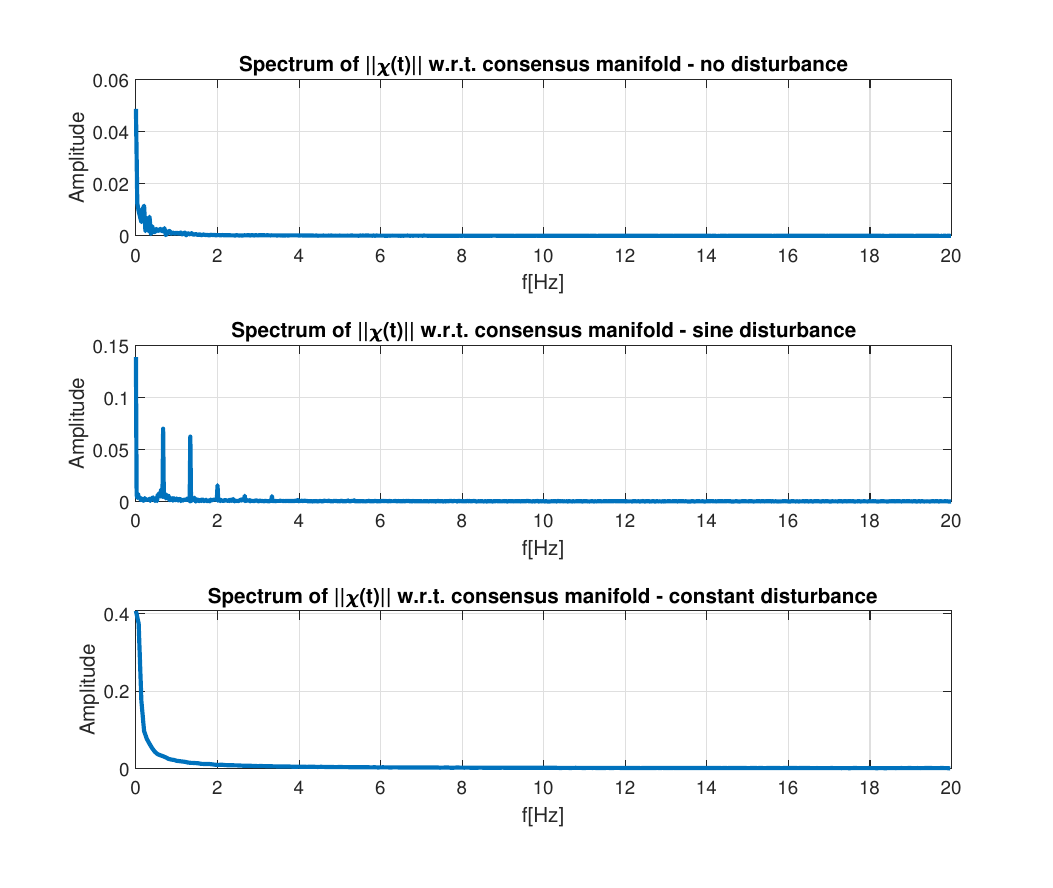}
        \caption{Amplitude spectra of signals of interest obtained via FFT.}
        \label{fig:amplitude_spectra}
    \end{figure}

\subsection{Spectral Network Identification}
\label{subsec:spectral_net_ident}

Network identification with few measurements (see Section \ref{subsec:network_ident_prob}) can be tackled through spectral methods. On the one hand, it is well-known from spectral graph theory that the eigenvalues of Laplacian matrices $\mathcal{L}^{\pm}$, see \eqref{eq:laplac}, (hereafter called \textit{Laplacian eigenvalues}) provide meaningful information on the topological structure of the graph, such as the mean vertex degree \cite{Chung1997}. Moreover, changes in the spectrum of the Laplacian matrix reveal changes in the graph structure, and communities can also be inferred from the Laplacian eigenvectors (see e.g. \cite{von2007tutorial}). On the other hand, the spectrum of the so-called Koopman generator
$$A_G:D(A_G) \to \mathcal{F},\quad A_Gf=\bar{F}_G\cdot\nabla f,$$
where $\mathcal{F}$ is a Banach space and $\bar{F}_G$ is a vector field describing a dynamical system over a network described by the graph $G$, can be estimated from data generated by the dynamics (see e.g. the Dynamic Mode Decomposition method \cite{Tu2014}). Also, it is noticeable that the data can be  measured at a few vertices (possibly one). Therefore, providing that the Laplacian spectrum can be retrieved from the Koopman operator spectrum, the global structure of the graph can be inferred through local measurements of the dynamics. 
\begin{rem}
In well-chosen spaces (depending on the dynamics), the operator $A_G$ is the infinitesimal generator of the strongly continuous semigroup of Koopman operators $(K^t)_{t\geq 0}: \mathcal{F} \to \mathcal{F}$ defined by the composition $(K^t)f(\cdot)= f \circ \varphi(t,\cdot)$ where $\varphi: \mathbb{R}^+ \times X \to X$ is the flow map generated by the dynamics \cite{Budisic2012,Koopman1931}. More details on semigroup theory can be found in Section \ref{mg_ss1}.
\end{rem}

From the above, it is clear that network identification with few measurements is equivalent to the \emph{spectral network identification} problem, defined as follows. Considering the set of all graphs $G$ with a fixed number of vertices, we aim at studying the existence of a correspondence between the set $\mathcal{S}_A$ of point spectra, $\sigma_p(A_G)$ of the Koopman generators $A_G$, and the set $\mathcal{S}_{\mathcal{L}^+}$ of spectra $\sigma(\mathcal{L}^+)$ of the in-degree Laplacian matrix $\mathcal{L}^+$ of the graphs $G$.


We now focus on the specific dynamical systems over network described by \eqref{eq:DSN_states}. The Jacobian matrix $\mathcal{J} \in \mathbb{R}^{np \times np}$ of the vector field at the origin is given by
\begin{equation*}
\mathcal{J} = \mathcal{I}_n \otimes \mathcal{B} - \mathcal{L}^+ \otimes \mathcal{D}
\end{equation*}
where $\mathcal{I}_n$ is the identity matrix, see \eqref{eq:Id&1}, $\mathcal{B}=\mathcal{J}_F \in \mathbb{R}^{p \times p}$ is the Jacobian matrix of $F$, and $\mathcal{D}=K(0) \mathcal{J}_H(0) \mathcal{J}_Q(0)$, with $\mathcal{J}_H \in \mathbb{R}^{r \times q}$ and $\mathcal{J}_Q \in \mathbb{R}^{q \times p}$ the Jacobian matrices of $H$ and $Q$, respectively. We assume that $\mathcal{J}$ has non-resonant eigenvalues $\mu_j$ ($j=1,\dots,np$) with $\Re\{\mu_j\}<0$, so that the origin is an asymptotically stable synchronised equilibrium. In this case, the point spectrum of the Koopman generator $A_G$ defined in the Hardy space $\mathcal{F}=H^2(\mathbb{D}^{np})$, where $\mathbb{D}^{np} \subset \mathbb{C}^{np}$ is a polydisk of a small enough radius, is given by 
\begin{equation*}
\sigma(A_G)=\left\{\sum_{j=1}^n \alpha_j \mu_j,\, \alpha\in \mathbb{N}^n\right\} \supset \sigma(\mathcal{J}).
\end{equation*}
It follows that there is a one-to-one correspondence between the set of spectra $\mathcal{S}_A$ and the set $\mathcal{S}_\mathcal{J}$ of spectra $\sigma(\mathcal{J})$, so that the problem boils down to an algebraic characterisation of the correspondence between the sets $\mathcal{S}_\mathcal{J}$ and $\mathcal{S}_{\mathcal{L}^+}$. We have the following result based on the spectral moments of the matrix $\mathcal{D}$ \cite{Gulina2022}.
\begin{prop}
Consider two in-degree Laplacian operators $\mathcal{L}_1^+$, $\mathcal{L}_2^+$ associated with two digraphs $G_1, G_1$. If $\mathrm{tr}(\mathcal{D}^k) \neq 0$ for all $k \in \{1,\dots,n\}$, then there is a one-to-one mapping between $\mathcal{S}_\mathcal{J}$ and $\mathcal{S}_{\mathcal{L}^+}$, that is
\begin{equation*}
                \sigma(\mathcal{J}_1) = \sigma(\mathcal{J}_2) \Leftrightarrow \sigma(\mathcal{L}^+_1) = \sigma(\mathcal{L}^+_2)
\end{equation*}
with $\mathcal{J}_1 = \mathcal{I}_n \otimes \mathcal{B} - \mathcal{L}^+_1 \otimes \mathcal{B}$ and $\mathcal{J}_2 = \mathcal{I}_n \otimes \mathcal{B} - \mathcal{L}^+_2 \otimes \mathcal{D}$.
\end{prop}
Moreover, the spectrum of $J$ can be easily expressed in terms of Laplacian eigenvalues.
\begin{lem}
\label{lem:spec_J}
The spectrum of $\mathcal{J}=\mathcal{I}_n \otimes \mathcal{B} - \mathcal{L}^+ \otimes \mathcal{D}$ is given by
\begin{equation*}
\sigma(\mathcal{J})  =\bigcup_{\lambda \in \sigma(\mathcal L^+)} \sigma(\mathcal{B}-\lambda \mathcal{D}).
\end{equation*}
\end{lem}
\begin{proof}
Let $\mathcal{P}$ be the matrix constructed with (generalised) eigenvectors of $\mathcal{L}^+$, such that $\mathcal{P}^{-1}\mathcal{L}^+\mathcal{P}$ has a Jordan form. Then the matrix $\tilde{\mathcal{J}}=(\mathcal{P} \otimes \mathcal{I}_p)^{-1} J (\mathcal{P} \otimes \mathcal{I}_p)=\textrm{diag}(\mathcal{M}_1,\dots,\mathcal{M}_r)$ is block-diagonal, with the blocks $\mathcal{M}_j=\mathcal{I}_{s_\lambda} \otimes \mathcal{B}-\mathcal{J}_{\lambda} \otimes \mathcal{D}$, where $\mathcal{J}_\lambda$ is a Jordan block associated with an eigenvalue $\lambda \in \sigma(\mathcal{L}^+)$, and $s_\lambda$ is its dimension. It is clear that the matrices $\mathcal{M}_j$ are block triangular, with diagonal blocks $\mathcal{B}- \lambda \mathcal{D}$. This implies that $\sigma(M_j) = \sigma(\mathcal{B}- \lambda \mathcal{D})$ and the result follows from the fact that $\sigma(\mathcal{J})=\sigma(\tilde{\mathcal{J}})=\cup_{j=1}^r \sigma(\mathcal{M}_j)$. 
\end{proof}
It follows from Lemma \ref{lem:spec_J} that the Laplacian eigenvalues $\lambda$ can be obtained from the eigenvalues $\mu \in \sigma(\mathcal{J}) \subset \sigma_p(\mathcal{A})$ by solving the generalised eigenvalue problem
\begin{equation*}
(\mathcal{B}-\mu \mathcal{I}_n)v=\lambda \mathcal{D} v.
\end{equation*}
For a given eigenvalue $\mu \in \sigma(\mathcal{J})$, there are as many solutions $\lambda$ as the rank of $\mathcal{D}$, but only one solution is a Laplacian eigenvalue. Specific techniques have to be developed to circumvent this issue (see \cite{Gulina2022} for more details).

When the matrix $\mathcal{D}$ has rank one (e.g. in the case $r=1$ or $q=1$), we have the following additional result \cite{Mauroy2017}.
\begin{prop}
Suppose that $\mathrm{rank}(\mathcal{D})=1$ so that there exists $v,w \in \mathbb{R}^m$ with $\mathcal{D}=v w^T$. If $\mathrm{rank}([v,\mathcal{B}v,\dots,\mathcal{B}^{p-1}v])=\mathrm{rank}([w,\mathcal{B}^Tw,\dots,(\mathcal{B}^{p-1})^T w])=p$, then there is a one-to-one mapping between $\mathcal{S}_J$ and $\mathcal{S}_{\mathcal{L}^+}$. Moreover,
\begin{equation*}
\sigma(\mathcal{L}) = \{0\} \cup \left\{ 1/w^T(\mathcal{B}-\mu \mathcal{I}_n)^{-1}v,\, \mu \in \sigma(\mathcal{J})\setminus \sigma(\mathcal{B})\right\}.
\end{equation*}
\end{prop}

The results presented here are restricted to dynamical systems over networks of the form \eqref{eq:DSN_states} which converge to a synchronized equilibrium. Further work could investigate the spectral identification problem for other dynamical systems over networks associated with different types of agents (e.g. oscillators) and coupling functions, and possibly exhibiting non-synchronized collective behaviors.


\section{PDEs on Metric Graphs}\label{sec:PDEs}

The notion of metric graph was introduced in Subsection \ref{sec:met_net}. In the current section we are going to investigate dynamic processes defined on this kind of structures.

\subsection{Hyperbolic Equation on Networks}
\label{mg_ss1}


The simplest hyperbolic equation models the advection or transport process on an interval. We consider it along the arcs of a metric graph $\mathcal{G}$. 
For simplicity, we rescale all edge lengths to 1. On an edge $e_j$, which we parameterise as $[0,1]$, we take the equation
\begin{equation}\label{eq:TE}
\frac{\partial}{\partial t}\, u_j(t,s) = c_j\cdot \frac{\partial}{\partial s}\, u_j(t,s)
,\quad t >0,\ s\in(0,1),\quad j\in J,
\end{equation}
where
the velocity coefficients are 
$c_{j}>0$ for all $j\in J$ and the flow is assumed to move in the direction opposite to the parametrisation, i.e., from endpoint 1 to endpoint 0.
Let us remark that one could also consider space- or time-variable coefficients $c_j$, as was done in \cite{EKF:22,KP:20,MS:07} or \cite{bkf-23}, respectively.
Throughout this section, denote
\[\mathcal{C}:=\diag({c_j}).\]
For the well-posedness of our problem, we should impose boundary conditions at the right endpoint of every edge, i.e., at 1. We shall assume that the weights $w_j$ on the edges of the underlying weighted combinatorial digraph $G$ (see Subsection \ref{sec:combi_net}) satisfy
\begin{equation}\label{eq:w}
0 < w_{j}\le 1 
\quad\text{and}\quad \sum_{e_j\in E^-_{v_i}} w_{j} = 1,
\end{equation}
for all $j\in J$ and $v_i\in V$. 
In particular, for the weighted outgoing degree matrix, we have $\mathcal{D}_w^-=I_n$.
The standard boundary conditions for our problem can be written as
\begin{equation}\label{eq:dis}
c_j u_{j}(t,1)= w_{j} \left[ \Phi^{+} \mathcal{C} u(t,0)\right]_i
\end{equation}
for every edge $e_j\in E^-_{v_i}$.
We assume without loss of generality that $G$ has no sinks or sources (see also \cite[Thm.~2.1]{BN:14}). Here and further on,
\[u(t,s)=\left(u_j(t,s)\right)_{j\in J}\]
denotes the vector of function values on the edges.

Note that conditions \eqref{eq:w} guarantee the conservation of mass in every vertex and together with condition \eqref{eq:dis} imply the so-called \emph{Kirchhoff's law} (“inflow = outflow”)  that can be formulated in terms of incidence matrices as
\begin{equation}\label{eq:kirch}
\Phi^{-} \mathcal{C} u(t,1) = \Phi^{+} \mathcal{C} u(t,0).
\end{equation}

The following result was proved in \cite[Prop.~2.5]{KS:05} and \cite[Prop.~18.15]{BKFR:17}.
\begin{prop}\label{prop:flow-wp}
Let $G$ be a finite connected digraph with no sinks or sources, given by the incidence matrices $\Phi^-$ and $\Phi^+$, and consider the system on the arches of the corresponding metric graph $\mathcal{G}$,
\begin{equation}\label{eq:F}
\left\{\begin{array}{rcll}
\frac{\partial}{\partial t}\, u_j(t,s) &=& c_j\cdot \frac{\partial}{\partial s}\, u_j(t,s)
,&t >0,\ s\in(0,1),\\
\phi_{ij}^{-}c_j u_{j}(t,1)&=&w_{ij}\sum_{k\in J}\phi_{ik}^{+}c_k u_{k}(t,0),& t >0,
\\
u_j(0,s)&=& f_j(s), &s\in\left[0,1\right],
\end{array}
\right.
\end{equation}
with coefficients $c_j>0$ and $w_{j}$ satisfying \eqref{eq:w}, for $j\in J$, $i\in I$. Then the problem is well-posed on the space $L^1\bigl([0,1],\CC^m\bigr).$
\end{prop}

\begin{rem}
The same result was in \cite[Prop.~2.3]{MS:07} and \cite[Cor.~2.19]{EKF:22} obtained also for non-constant coefficients $c_j$.
Moreover, problem \eqref{eq:F} is well-posed also on $L^p\bigl([0,1],\CC^m\bigr)$ for $1\leq p<+\infty$, see \cite[Cor.~2.19]{EKF:22} and \cite[Prop.~3.1]{KP:20}.
\end{rem}

Recall that an abstract Cauchy problem on a Banach space $X$ associated with a closed operator $A\colon D(A)\subset X\to X$,
\begin{equation*}
\text{(ACP)} \begin{cases}
\frac{d}{dt}{x}(t)=Ax(t),& t>0,\\ 
x(0)=f\in X
\end{cases} \end{equation*}
is \emph{well-posed} if and only if $(A, D(A))$ is the (infinitesimal) \emph{generator} of a \emph{strongly continuous semigroup} $\Tt$ on $X$, cf.~\cite[Cor.~II.6.9]{EN:06}. In this case, the mild (or, in the case $f\in D(A)$, classical) solution to (ACP) is given by $x(t)=T(t)f$, $t\geq 0$, and many properties of the solution can be inferred from the respective characteristics of the semigroup, which are frequently derived from the spectral properties of the generator.
For further information we refer to \cite{EN:06}.

Problem \eqref{eq:F} can be rewritten in the form of an abstract Cauchy problem on $L^1\bigl([0,1],\CC^m\bigr)$ by defining the 
 operator
\begin{equation}\label{eq:flow-op}
A_F := \mathcal{C}\cdot\frac{d}{ds},\quad
D\left(A_F\right) :=\left\{f\in\mathrm{W}^{1,1}\bigl([0,1],\CC^m\bigr) \colon f(1) = \mathcal{B}^+_C f(0) \right\}
\end{equation}
(for the condition of the domain see \cite[Prop.~18.2]{BKFR:17}), where $\mathcal{B}^+_C$ is the weighted adjacency matrix of the line graph obtained as 
\begin{equation*}
\mathcal{B}_C^+:= \mathcal{C}^{-1} \mathcal{B}_w^+ \mathcal{C}.
\end{equation*}
Proposition \ref{prop:flow-wp} tells us that $(A_F, D(A_F))$ is the generator of a strongly continuous semigroup $\Tt$ on $L^1\bigl([0,1],\CC^m\bigr)$ and the solution to \eqref{eq:F} is given by
\[u_{f}(t)= T(t)f,\quad t\geq 0,\] where $u_f(t)=u_f(t,\cdot)$ denotes the vector of solution functions on the edges with an initial value $f$.

Under the assumption on the weights \eqref{eq:w}, the matrix $\mathcal{B}_w^+$ is column stochastic. This turns out to be important when studying further qualitative properties of the solutions. 
Many properties of the semigroup - and hence of the solutions - are given by the structure of the graph. For example, the semigroup $\Tt$ is irreducible if and only if the oriented graph $G$ is strongly connected (cf.~\cite[Prop.~18.16]{BKFR:17} and \cite[Lem.~4.5]{MS:07}).

The spectrum of the operator $(A_F, D(A_F))$ can be characterized using the weighted adjacency matrix $\mathcal{A}_w^{+}$, the weighted adjacency matrix of the line graph $\mathcal{B}_C^+$ or, equivalently, the advection matrix $\mathcal{N}^{+}_w$ as follows. We start by adjusting the nonzero coefficients of these matrices. Define 
\begin{equation}
E_{\lambda}(s):=\diag\left(e^{\frac{\lambda}{c_{k}}s } \right),\quad s\in \mathbb{R},\label{eq:te_epl}
\end{equation}
and let
\[
\mathcal{A_{\lambda}^+} := \Phi ^{+} E_{\lambda}(-1)\left( \Phi_w^{-}\right) ^{\top}, \quad \mathcal{B}_{C,\lambda}^+:= E_{\lambda}(-1)\mathcal{B}_{C}^+,\quad\text{and}\quad \mathcal{N_{\lambda}^+} :=I-\mathcal{A_{\lambda}^+}.
\]
Notice that $\mathcal{A}_0^+ = \mathcal{A}_w^+$, $\mathcal{B}_{C,0}^+ = \mathcal{B}_{C}^+$, and $\mathcal{N}_0^+ = \mathcal{N}^+_w$ (recall that $\mathcal{D}_w^-=I_n$ by \eqref{eq:w}).

\begin{thm}\label{thmEszterMarjeta}
Let $(A_F, D(A_F))$ be the operator given in \eqref{eq:flow-op}. Then 
\[\lambda\in\sigma(A_F)=\sigma_p(A_F)\Longleftrightarrow 1\in\sigma(\mathcal{A}_{\lambda}^+) \Longleftrightarrow 1\in\sigma(\mathcal{B}_{C,\lambda}^+)\Longleftrightarrow 0\in\sigma(\mathcal{N}_{\lambda}^+). \] 
Moreover, $g$ is an eigenfunction of $A_F$ corresponding to the eigenvalue $\lambda$ if and only if $g(s)=E_{\lambda}(s)d$ where $d$ is an eigenvector of $\mathcal{B}_{C,\lambda}$ corresponding to the eigenvalue 1.
\end{thm}
\begin{proof}
The equalities about the spectra follow from the explicit formula of the resolvent $R(\lambda,A_F)$ which is positive and compact, see \cite[Cor.~18.13]{BKFR:17}. Observe that the eigenspace $\ker (\lambda-A_F)$ is spanned by the functions of the form $g(\cdot) = E_{\lambda}(\cdot) d$ for some $d\in \mathbb{C}^{J}$ such that $g(1) = \mathcal{B}^+_C g(0)$. The given correspondence between eigenvectors can be derived by an easy calculation. 
\end{proof}
\medskip

The following condition plays a crucial role in the long-term behaviour of the solutions.
\begin{equation}\label{ldq}
\begin{aligned}
\text{There exists } &0<d\in\RR \text{ such that }d\cdot \left(\frac{1}{c_{j_1}} + \cdots + \frac{1}{c_{j_k}}\right)\in\NN \\ 
&\text{for all directed cycles }e_{j_1},\dots,e_{j_k} \text{ in }G.
\end{aligned}
\end{equation}
If \eqref{ldq} is satisfied, 
the boundary spectrum of $A_F$ is cyclic (see \cite[Prop.~18.1b]{BKFR:17}). Moreover, in this case also
a so-called \emph{circular spectral mapping theorem} holds for the spectrum of the generator $(A_F, D(A_F))$ and the semigroup $\Tt$, see \cite[Prop.~3.8]{KS:05} and \cite[Prop.~4.10]{MS:07}. This states the following:
\[\Gamma\cdot e^{t\sigma(A_F)}=\Gamma\cdot\sigma (T(t))\setminus\{0\}\text{ for each }t\geq 0,\]
where $\Gamma$ denotes the unit circle.

From the spectral properties of the generator we can obtain the asymptotic behaviour of the solutions of problem \eqref{eq:F}, see \cite[Thm.~4.10]{KS:05} and \cite[Thm.~18.19]{BKFR:17}. 
We call a subgraph $G_r$ of $G$ a \emph{terminal strong component} if it is strongly connected and there are no outgoing edges of $G_r$, see \cite[page 17]{BaJeGu:01}.

\begin{thm}
\label{thm:flow-asy} 
Let $G$ be a connected graph with terminal strong components $G_{1},$ $\dots,$ $G_{{\ell}}$ and $\Tt$ the semigroup associated with the transport problem \eqref{eq:F}. Then for any initial value $f\in L^1\bigl([0,1],\CC^m\bigr)$ the (mild) solution of system \eqref{eq:F} can be written uniquely in the form of a sum
\begin{align*}
u_f(t)=T(t)f&=T_{nil}(t)f + T_{stab}(t)f + T_{r_1}(t)f + \cdots + T_{r_{\ell}}(t)f,\\
&\coloneqq u_{nil,f}(t)+u_{stab,f}(t)+u_{r_1,f}(t) + \cdots + u_{r_{\ell},f}(t),\quad t\geq 0,
\end{align*}
where $T_{nil}(\p)$, $ T_{stab}(\p)$, $ T_{r_1}(\p)$, $ \dots,$ $T_{r_{\ell}}(\p)$ are strongly continuous semigroups on the $T(t)$-invariant subspaces $X_{nil}$, $X_{stab}$, $X_{r_1}$, $\cdots$, $X_{r_{\ell}}$, respectively, with
\[L^1\bigl([0,1],\CC^m\bigr)=X_{nil} \oplus X_{stab} \oplus X_{r_1} \oplus \cdots \oplus X_{r_{\ell}},\]
and the following holds:
\begin{enumerate}
\item There exists $t_0>0$ such that $u_{nil,f}(t)=0$ for all $t\geq t_0$.
\item The solution on $X_{stab}$ is stable: $\lim\limits_{t\to+\infty}u_{stab,f}(t)=0$.
\item If for some $1\le i\le\ell$ the graph $G_{i}$ satisfies Condition \eqref{ldq} then for the period
\[ \tau_i = \frac{1}{d}\gcd\left\{ d\cdot\left( \frac{1}{c_{j_1}} + \cdots + \frac{1}{c_{j_k}}\right) \colon e_{j_1},\dots,e_{j_k} \text{ is a directed cycle in }G_i\right\},\]
the solution on $X_{r_i}$ behaves periodically: 
\[u_{r_i,f}(t+\tau_i)=u_{r_i,f}(t), \quad t\in\RR.\]
\item If for some $1\le i\le\ell$ graph $G_{i}$ does not satisfy Condition \eqref{ldq}, then for the solution on $X_{r_i}$ it holds
\[\lim_{t\to+\infty}u_{r_i,f}(t)=\mathcal{P}_{X_{r_i}}f,\] where $\mathcal{P}_{X_{r_i}}$ denotes the projection onto the one-dimensional subspace $X_{r_i}$.
\end{enumerate}
\end{thm}

Therefore, the structure of the discrete graph strongly influences the long-term behaviour of the solutions to the given problem. 

Let us also note that one can generalize the problem and consider a system of hyperbolic equations on each graph edge. Such problems were studied in \cite{KMN:21} in the $L^2$-setting where the existence of the solutions was obtained under appropriate assumptions for the (not necessarily hermitian) coefficient matrices and admissible boundary conditions.

We further mention that stabilization results of some hyperbolic systems on graphs are available in \cite{nic:hyperbolic,Hayeketall-20}, where exponential or polynomial energy decay is obtained.

Another class of problem \eqref{eq:F} generalizations is transport on time-varying metric graphs. Depending on the types of graph modifications, one can mention \cite{Bayazit_2013} where time-dependent weights are considered, \cite{bkf-23} with both weights and velocities of flow differing in time, and finally \cite{LNP_2024} where not only the coefficients but also the structure itself can be modified.

\subsection{Virus Variant Modelling in Fragmented Environment}\label{subsec:virus2}

Let us now return to the motivating example in Subsection \ref{subsec:virus}. We build the mathematical model based on network transport equation that describes the evolution of variants of the virus. The aim of this example is answering the following research problem: What type of vaccine should be widely available in a chosen area, to optimise the effectiveness of vaccination policy?

We consider $P$ patches that can be interpreted either as cities or larger areas divided by natural barriers such as rivers, mountain ranges, etc. 
In each patch $p$, $p=1,\ldots, P$, there are people colonised by exactly one variant $q$ of the virus, $q=1,\ldots,Q-1$, as well as uncolonised persons. For the sake of simplicity, we denote the uncolonised persons by colonised by virus variant $q=Q$. Additionally, we denote by $t_{q}$, $q=1,\ldots,Q-1$, an average duration of infection caused by the $q$-th virus variant. Similarly, we choose $t_{Q}$ as the time during which uncolonised persons remain uncolonised.

We start the description by defining a combinatorial graph $G$. Let us associate vertex $v_{(q-1)P+p}\in V=\{v_i:\,i=1,\ldots,PQ\}$ with a group of people in the patch $p=1,\ldots,P$, having the virus variant $q=1,\ldots,Q$. Edges indicate the intensities of both virus transfers between patches and the emergence of mutations. The edge connecting $v_i$ and $v_j$, $i,j=1,\ldots, PQ$, has a weight denoted by $w_{ji}$, which can be interpreted in the following way:

\begin{itemize}
\item fraction of mutations from the $q_2$-th variant to the $q_1$-th variant in patch $p$
$$w_{(q_1-1)P+p,(q_2-1)P+p};\qquad p=1,\ldots,P,\,\, q_1,q_2=1,\ldots, Q;$$

\item fraction of patients colonised by the $q$-th variant traveling from patch $p_2$ to $p_1$, given by:
$$w_{(q-1)P+p_1,(q-1)P+p_2},\qquad p_1,p_2=1,\ldots,P,\,\, q=1,\ldots,Q;$$ 

\item fraction of patients colonised by the $q$-th variant suffering longer than the average time $t_{q}$ in patch $p$, given by:
$$w_{(q-1)P+p,(q-1)P+p};\qquad p=1,\ldots,P,\,\, q=1,\ldots, Q.$$ 
\end{itemize}
Assuming that mutations appear only in patches, for $p_1\neq p_2$, $q_1\neq q_2$, $p_1,p_2=1,\ldots, P$, $q_1,q_2=1,\ldots,Q$ we have 
$$w_{(q_1-1)P+p_1,(q_2-1)Q+p_2}=0.$$

We divide the dynamic process describing the spread of infection in society into two stages. The first is passing through the life-cycle of infection while the second is an instantaneous transfer of patients between patches, an instantaneous mutation of the virus, or an instantaneous prolongation of colonisation. 

The second dynamical phenomenon can be associated with edges' weights on the combinatorial graph $G$. In order to introduce the first dynamical process we extend each vertex into interval $[0,1]$, hence locally one-dimensional metric space, and define a transport equation on a metric graph $\mathcal{G}$.  Assuming that $[0,1]$ is parametrised with $s$, then position $s\in [0,1]$ in vertex $v_i$ determines the time $st_i$ that has elapsed since the colonisation. Consequently, $s$ can be regarded as a local time parameter while $t$ as the global one. 

In this way, one can identify the flow along the interval $[0,1]$ with an evolution of population distribution assigned to this particular vertex. Namely, we define a function $u(t,s)=(u_i(t,s))_{i=1,\ldots,QP}$ such that $u_i(t,s)$ describes the distribution of persons at time $t$ that were infected at $t-1+s$. We obtain a system \eqref{eq:F} with $\Phi^{\pm}$ being incidence matrices of graph $G$ and 
\begin{equation*}
c_{(q-1)P+p}=t_q^{-1},\textrm{for } p=1,\ldots, P; \,\, q=1,\ldots, Q,
\end{equation*}
the velocities of flow. Due to the inaccuracy of parameter estimation, we assume without loss of generality that $t_{q}\in \mathbb{Q}_+$, $q=1,\ldots,Q$, and consequently, \eqref{ldq} always holds.

Let us refer first to the network structure of this model. On the one hand we interpret the problem as dynamics in (one-dimensional) vertices defined by a system of PDEs. On the other, one can build a metric graph $\mathcal{G}$ such that $L(\mathcal{G})=G$. If it is possible, the problem can be associated with transport on the metric graph $\mathcal{G}$, and we say that the problem is graph-realisable, \cite{BF:15}. The graphical representation of all three objects, namely combinatorial graph $G$, metric graph's counterpart $\mathcal{G}$ and combinatorial graph with locally one-dimensional vertices is presented in Figure \ref{fig:3networks}, in the case $P=1$, $Q=2$. In the following considerations, we apply metric graph theory hence we use the nomenclature related to metric graphs.

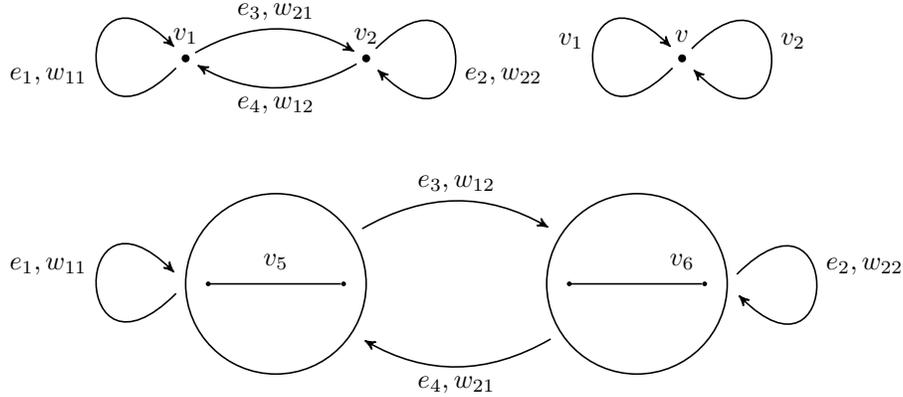
\begin{figure}
\centering
\begin{tikzpicture}[scale=0.6, >=stealth',shorten >=1pt,auto,node distance=1.5cm,  semithick]
\tikzstyle{every state}=[fill=none,draw=none,text=black]

\draw (3,-4) circle (2cm);
\draw (11,-4) circle (2cm);
\filldraw (1,1) circle (2pt);
\filldraw (5,1) circle (2pt);
\filldraw (12,1) circle (2pt);

   \node (p1) at (1, 1) {};
   \node (A) at (1,1.5) {$v_1$};
   \node (p2) at ( 5, 1) {};
   \node (B) at (5,1.5) {$v_2$};
   \node (p3) at ( 12, 1) {};
   \node (C) at (12,1.5) {$v$};
   \node (p4) at ( 1, -4) {};
   \node (p5) at ( 3, -4) {};
    \node (D) at (3,-3.5) {$v_5$};
   \node (p6) at ( 11, -4) {};
   \node (E) at (12,-3.5) {$v_6$};
   \node (p7) at ( 13, -4) {};
   \node (p8) at ( 4.7, -2.9) {};
   \node (p9) at ( 9.3, -2.9) {};
   \node (p10) at ( 4.7, -5.1) {};
   \node (p11) at (9.3, -5.1) {};

\draw (1.5,-4) -- (4.5,-4);
\draw (9.5,-4) -- (12.5,-4);
\filldraw (1.5,-4) circle (1pt);
\filldraw (4.5,-4) circle (1pt);
\filldraw (9.5,-4) circle (1pt);
\filldraw (12.5,-4) circle (1pt);
   
\path[->,every loop/.style={looseness=20}]      (p3) edge      [in=-45,out=45,loop]       node [above right] {$v_2$} (p3);
\path[->,every loop/.style={looseness=20}]      (p3) edge      [in=135,out=-135,loop]       node [above left] {$v_1$} (p3);

\path[->,every loop/.style={looseness=20}]      (p1) edge      [in=135,out=-135,loop]       node [below left] {$e_1, w_{11}$} (p1);
\path[->,every loop/.style={looseness=20}]      (p2) edge      [in=-45,out=45,loop]      node [below right] {$e_2, w_{22}$} (p2);
\path[->]     (p1) edge      [bend left]       node [above]{$e_3,w_{21}$} (p2);
\path[->]      (p2) edge      [bend left]       node [below]{$e_4,w_{12}$} (p1);

\path[->,every loop/.style={looseness=20}]     (p4) edge      [in=135,out=-135,loop]       node [above left] {$e_1, w_{11}$} (p4);
\path[->,every loop/.style={looseness=20}]     (p7) edge      [in=-45,out=45,loop]       node [above right] {$e_2, w_{22}$} (p7);
\path[->]      (p8) edge      [bend left]       node [above]{$e_3,w_{12}$} (p9);
\path[->]      (p11) edge      [bend left]       node [below]{$e_4,w_{21}$} (p10);

\end{tikzpicture}
\caption{Three network structures representing the spread of two virus variants in one patch (namely $P=1$, $Q=2$). In the upper left a combinatorial graph $G$, in the upper right a metric graph $\mathcal{G}$, such that $L(\mathcal{G})=G$. In the lower picture a combinatorial graph with
locally one-dimensional vertices.}\label{fig:3networks}
\end{figure}

Based on Theorem~\ref{thm:flow-asy}, we estimate the influence of  the $q$-th virus variant in $p$-th patch. Therefore, we look at the dynamics on edge $e_{(q-1)P+p}$, $p=1,\ldots, P$, $q=1,\ldots,Q$, of metric graph $\mathcal{G}$ (equivalently in $v_{(q-1)P+p}$  vertex of $G$). 
If this edge does not belong to any terminal strong component $G_r$, we have two possible scenarios. In the first one, the $(q-1)P+p$-th coordinate of a function $u_{stab,f}$ vanishes for any $t\geq 0$. Consequently, the $q$-th variant of a virus will become extinct in the $p$-th patch in a predictable time horizon. Otherwise, the $q$-th variant will be vanishing gradually from the $p$-th patch. Consequently, in both presented cases, the $q$-th variant should not be taken into account for vaccination policies in patch~$p$.

Now, let us consider an edge $e_{(q-1)P+p}$ that belongs to a terminal strong component $G_r$. By Theorem~\ref{thm:flow-asy}, the densities on edges from $G_r$ change periodically, with period $\tau_r$, hence one can compute the average number of patients colonised by any variant of a virus in the $r$-th terminal strong component. If $f$ is an initial distribution of patients in all patches, then we define a mean $\Pi^r: L^1\bigl([0,1],\CC^m\bigr) \rightarrow \RR^m$
$$\Pi^r f=\frac{1}{\tau_r}\int_{0}^{\tau_r}\int_0^1 (T_r(s)f)(x)dx ds, \quad \textrm{for any}\,\, t\geq 0.$$

In the last step, we compare the average proportion of persons colonised by the $q$-th virus variant with the average proportion of all persons colonised by viruses enumerated by $q=1,\ldots,Q-1$ in the patch. The vector of averaged prevalence $\textrm{Prev}_f(p)=(\textrm{Prev}_{q,f}(p))_{q=1,\ldots,Q-1}$ in the $p$-th patch, with an initial distribution $f$, is given by
$$\textrm{Prev}_{q,f}(p):=\frac{\left(\sum_{i=1}^{\ell}\Pi^{r_i}f\right) \mathbf{1}_{(q-1)P+p}}{\left(\sum_{i=1}^{\ell}\Pi^{r_i}f\right)\,\left(\sum_{k=1}^{Q-1}\mathbf{1}_{(k-1)P+p}\right)},\quad q=1,\ldots,Q-1,$$
where $\mathbf{1}_{x}\in\RR^m$ is a vector having only one non-zero entry at the $x$-th coordinate. The largest value among all coordinates of  $\textrm{Prev}_f(p)$ indicates the virus variant that dominates in the $p$-th patch. Consequently, the authorities should provide the vaccine in patch $p$ that has the highest efficacy against that particular variant. 

\subsection{Laplace Equation on Graphs}
\label{ssLaplaceongraphs}

Second-order equations on graphs have generated significant interest among mathematicians in fields like partial differential equations, control theory and spectral theory, as well as among physicists, leading to a large body of literature. Such problems arise naturally when one considers heat conduction or propagation of waves through a one-dimensional system that can be modeled by a graph. Further applications are quantum wires, photonic crystals, carbon nano-structures, etc. Most of the early results are about the spectrum of the Laplacian on graphs, see e.g.~\cite{nicaise:85}, while \cite{KMS:07} is the first to deal with the asymptotic behavior of solutions to the diffusion problem.

We now consider the diffusion process on an interval along the arcs of a metric graph $\mathcal{G}$. 
We adopt the following notation for function $u_j$, defined on the edge $e_j$, parameterized on $[0,1]$:
\begin{equation*}
u_j(v_i)=\left\{\begin{array}{ll}
u_j(0)&\textrm{if } v_i\stackrel{e_j}{\rightarrow},\\
u_j(1)& \textrm{if } \stackrel{e_j}{\rightarrow}v_i.\\
\end{array}\right.
\quad
\frac{d}{ds}u_j(v_i)=\left\{\begin{array}{ll}
\frac{d}{ds}u_j(0)&\textrm{if } v_i\stackrel{e_j}{\rightarrow},\\
\frac{d}{ds}u_j(1)& \textrm{if } \stackrel{e_j}{\rightarrow}v_i,\\
\end{array}\right.
\end{equation*}
if, in the latter case, (one-sided) derivatives exist.
Let us first recall the definition of the (continuous) Laplacian $\Delta_{\mathcal{G}}$ on a metric graph $\mathcal{G}$ with Kirchhoff boundary conditions at the vertices. In the following, $C(\mathcal{G})$ denotes the space of continuous functions on the edges being continuous across the vertices as well, that is,
\[C(\mathcal{G}):=\{u=(u_{j})_{j\in J}\in \prod_{j\in J}C[0,1]:u_j(v_i)=u_k(v_i)\text{ if }\phi_{ij}\phi_{ik}\neq 0\},\]
where
\[\Phi=(\phi_{ij})_{i\in I, j\in J}=\Phi^+ -\Phi^-\]
denotes the incidence matrix of the underlying graph $G$.
Let
\begin{equation}\label{eq:LaplC_dom}
D(\Delta_\mathcal{G}):=\{u\in C(\mathcal{G}): u_{j}\in H^2(0, 1), \forall j\in J,\; 
\sum_{l\in J} \phi_{il} \frac{d}{ds}u_l(v_i)=0, \forall i\in I\}
\end{equation}
and
\begin{equation}\label{eq:LaplC}
\Delta_{\mathcal{G}} u:=\left(-\frac{d^2}{ds^2} u_{j}\right)_{j\in J}.
\end{equation}

It is not difficult to show that $\Delta_{\mathcal{G}}$ is a non-negative self-adjoint operator with a compact resolvent.

We first recall an old result from \cite{below:85,nicaise:85} in the same spirit as Theorem \ref{thmEszterMarjeta}. In the case when the lengths of the edges are all equal to 1, it gives the relationship between the spectrum of the operator $\Delta_{\mathcal{G}}$ and the eigenvalues of the so-called  \emph{transition matrix} of $G$, that is the matrix $\widetilde{\mathcal{A}}$ , defined as
\[
\widetilde{\mathcal{A}}:=\mathcal{D}^{-1}\mathcal{A},
\]
where $\mathcal{A}$ is the adjacency matrix of $G$, see \eqref{eq:unwundiradjacency}, and $\mathcal{D}$ is the degree matrix of $G$, see \eqref{eq:unwdegree}.

\begin{thm}
Assume that the lengths of the edges of the metric graph $\mathcal{G}$ are all equal to 1. Then
the spectrum $\sigma(\Delta_{\mathcal{G}})$ consists only of eigenvalues and is given by
\[
\sigma(\Delta_{\mathcal{G}})=S_1\cup S_2,
\]
where 
\[
S_1=\{\pi^2 k^2: k\in \mathbb{N}\},
\]
the multiplicity of 0 is 1, while the multiplicity of $\pi^2 k^2$ depends on the fact of whether $G$ is bipartite or not. Furthermore, 
\[
S_2=\{\lambda\in (0,\infty): \cos\sqrt{\lambda}\in \sigma(\widetilde{\mathcal{A}})\cap(-1,1)\}.
\]
\end{thm}
Let us notice that the main idea for the characterization of the family $S_2$ is to look for
an eigenvector $u$ in the form
\[u_j(s)=\frac{u_j(0) \sin (\sqrt{\lambda} (1-s))+u_j(1) \sin (\sqrt{\lambda} s)}{\sin \sqrt{\lambda}}, \quad s\in [0,1],\]
which is meaningful since $\sin \sqrt{\lambda}\neq 0$. By the continuity condition at the vertices, the remaining unknowns are the values of $u$
at the vertices and we easily check that the Kirchhoff conditions at the nodes lead to
\[
\widetilde{\mathcal{A}} X=\cos\sqrt{\lambda} X,
\]
where 
\begin{equation}\label{eq:vertexvalues}
 X=(u(v))_{v\in V}
\end{equation}
denotes the vertex values of $u$ which are well-defined for a function $u\in C(\mathcal{G})$.

This method also yields an explicit relationship between the eigenvector $u$ of $\Delta_{\mathcal{G}}$ associated with $\lambda$
and the eigenvector $X$ of $\widetilde{\mathcal{A}}$ associated with $\cos\sqrt{\lambda}$.

In the following, we consider a diffusion problem on $\mathcal{G}$ of the form
\begin{equation}\label{eq:diffusiongraph}
\left\{\begin{array}{rcll}
\frac{\partial}{\partial t}u_j(t,s) &=& \frac{\partial^2}{\partial s^2}u_j(t,s)
,&t >0,\ s\in(0,1),\\
u_j(v_i)&=&u_k(v_i)&\text{if }\phi_{ij}\phi_{ik}\neq 0,\\
0&=&\sum_{l\in J}\phi_{il}\frac{d}{ds}u_{l}(t,v_i),& t >0,
\\
u_j(0,s)&=& f_j(s), &s\in\left[0,1\right],
\end{array}
\right.
\end{equation}
where $j,k\in J$, $i\in I$. It can be easily seen that the above system is equivalent to the abstract Cauchy problem
\begin{equation}\label{eq:diffG}
\begin{cases}
\frac{d}{dt}u(t)=-\Delta_{p,\mathcal{G}} u,\\
 u(0)=f,
\end{cases}
\end{equation}
with state space $L^p\bigl([0,1],\CC^m\bigr)$, where
\begin{eqnarray}\label{eq:Lapl-Lp-dom}
\begin{array}{ll}
D(\Delta_{p,\mathcal{G}})=
\left\{u\in C(\mathcal{G}):\right. &u_{j}\in W^{2,p}(0, 1),\,\forall j\in J \\
&\left.\hbox{and } \sum_{l\in J} \phi_{il} \frac{d}{ds} u_l(v_i)=0,\, \forall i\in I\right\}\end{array}
\end{eqnarray}
and
\begin{equation}\label{eq:Lapl-Lp}
\Delta_{p,\mathcal{G}} u=\left(-\frac{d^2}{ds^2} u_{j}\right)_{j\in J}.
\end{equation}

In \cite[Thm.~3.6]{KMS:07}, the following result is shown, see also \cite[Cor.~2.13]{EKF:19} and \cite[Prop.~3.3]{KP:20}.
\begin{thm}
The first-order problem \eqref{eq:diffG} is well-posed on the space $L^p\bigl([0,1],\CC^m\bigr)$ for $1\leq p<+\infty$, i.e., for all initial data $f\in L^p\bigl([0,1],\CC^m\bigr)$ problem \eqref{eq:diffG} admits a unique mild solution that continuously depends on the initial data. The solutions on the appropriate spaces have the form 
\begin{equation}\label{eq:diffusiongraphsol}
 u_{p,f}(t)=T_p(t)f,\quad t\geq 0,
\end{equation} 
where $\left(T_p(t)\right)_{t\geq 0}$ is the strongly continuous semigroup generated by $-\Delta_{p,\mathcal{G}}$ on $L^p\bigl([0,1],\CC^m\bigr)$.
\end{thm}

By \cite[Cor.~5.2]{KMS:07}, we can also describe the asymptotic behaviour of the solutions to \eqref{eq:diffG}. 
\begin{thm}
Let $f\in L^p\bigl([0,1],\CC^m\bigr)$, $1\leq p<+\infty$, be arbitrary. For solutions of problem \eqref{eq:diffG} with the initial value $f$ the following holds:
\begin{enumerate}
\item There exists a limit $\lim\limits_{t\to+\infty}u_{p,f}(t)=\mathcal{P}f$.
\item $\mathcal{P}$ is the strictly positive projection onto the one-dimensional subspace of $L^p\bigl([0,1],\CC^m\bigr)$ spanned by the constant function $1$, which is the kernel of $\Delta_{p,\mathcal{G}}$.
\item For every $\varepsilon>0$ there exists $M>0$ such that
\[\|u_{p,f}(t)-\mathcal{P}f\|\leq M\mathrm{e}^{(\varepsilon-\lambda_2)t},\]
where $\lambda_2$ is the smallest positive eigenvalue of $\Delta_{\mathcal{G}}$.
\end{enumerate}
\end{thm}

Note that the situation here essentially differs from the asymptotic behavior for the transport equation presented in Theorem \ref{thm:flow-asy}. In the case of transport, the behavior was heavily based on condition \eqref{ldq}, which encompasses the structure of the graph and the rational dependence of the corresponding velocity coefficients. In the case of diffusion, we always observe convergence towards projection to a one-dimensional subspace, independently of the structure or possible diffusion coefficients on the edges of the graph (for simplicity, we have taken all the coefficients equal to 1 but the asymptotic result holds even in the case of different time-dependent coefficients, see \cite[Thm.~5.4]{ADK:14}).

In \cite{KS23}, the Kirchhoff boundary condition of the diffusion problem \eqref{eq:diffusiongraph} in each vertex were perturbed by noise, that is, the following system was considered:
\begin{equation}\label{eq:diffusiongraphnoise}
\left\{\begin{array}{rcll}
\frac{\partial}{\partial t}u_j(t,s) &=& \frac{\partial^2}{\partial s^2}u_j(t,s)
,&t >0,\ s\in(0,1),\\
u_j(v_i)&=&u_k(v_i)&\text{if }\phi_{ij}\phi_{ik}\neq 0,\\
\dot{\beta}_{v_i}(t) &=&\sum_{l\in J}\phi_{il}\frac{d}{ds}u_{l}(t,v_i),& t\in (0,T],
\\
u_j(0,s)&=& f_j(s), &s\in\left(0,1\right),
\end{array}
\right.
\end{equation}
where $j,k\in J$, $i\in I$ and 
\[(\beta(t))_{t\in [0,T]}=\left(\left(\beta_{v_i}(t)\right)_{t\in [0,T]}\right)_{i\in I},\] 
is an $\RR^{n}$-valued Brownian motion (Wiener process) on an appropriate complete probability space. For an initial function $f\in L^2\bigl([0,1],\CC^m\bigr)$, the mild solution of problem \eqref{eq:diffusiongraphnoise} has the form
\begin{equation}
 \tilde{u}_{f}(t)=u_{2,f}(t)+\int_0^t (\lambda+\Delta_{\mathcal{G}}) S(t-s)D_{\lambda}\, d\beta (s),\quad t\geq 0,
\end{equation}
where $u_{2,f}$ is the solution of the diffusion problem on $L^2$ from \eqref{eq:diffusiongraphsol}, $\lambda>0$ is arbitrary, and $D_{\lambda}$ is the so-called Dirichlet-operator. In \cite[Thm.~3.8]{KS23}  the following result was proved.

\begin{thm}
 Let $(f_k)$ be a complete orthonormal system consisting of the eigenfunctions of $\Delta_{\mathcal{G}}$ in $L^2\bigl([0,1],\CC^m\bigr)$, and assume that there exists a positive constant $c$ such that 
 \[\sup\{|f_k(v)|:v\in V,\, k\in \NN\}\leq c.\]
Then, for $\alpha<\frac{1}{4}$ the mild solution $\tilde{u}_{f}$ has a continuous version in the fractional domain space $D\left(\Delta_{\mathcal{G}}^{\alpha}\right)$.
\end{thm}

\section{PDEs on Embedded Metric Graphs}\label{sec:embed_PDEs}


\subsection{Traffic Analysis}\label{subsec:traffic}

An example application of metric graphs that are embedded into the plane is traffic models, in particular macroscopic ones, which are usually based on partial differential equations describing the evolution of aggregated values such as traffic density, traffic flow, and average speed. 
In this article, we describe the study \cite{vandissel2024paynewhitham} of Cartier van Dissel, Gora and Manea, focusing on one of the popular macroscopic traffic simulation models, the Payne-Whitham model \cite{Payne1971,Whitham1999}, originally designed for modelling highway traffic on straight road segments.
In \cite{vandissel2024paynewhitham}, a new model is derived from the one in \cite{KotsialosPapageorgiou}, with the main novelty being that some of the road segments are controlled by traffic signals. In particular, the study focuses on how drivers react to the change of state of a traffic light and improves the model from \cite{KotsialosPapageorgiou} by taking into account the reduction of speed when cars enter a junction and turn left or right.

The Payne-Whitham model on a single edge (i.e., road segment) consists of a conservation law for density, together with a characterisation of the instantaneous variation of speed based on the behaviour of drivers in terms of their adaptation to the current density and also their ability to anticipate the changes in density upstream. The equations are the following:
\begin{align}
\label{eq:density_continuouos}
&\frac{\partial \rho(t,s)}{\partial t} +\frac{\partial(v(t,s)\rho(t,s))}{\partial s}=0\\
\label{eq:speed_continuouos}
&\frac{\partial v(t,s)}{\partial t}+v(t,s)\frac{\partial(v(t,s))}{\partial s} = \frac{V(\rho(t,s))-v(t,s)}{\nu} - \frac{C}{\rho(t,s)} \frac{\partial \rho(t,s)}{\partial s},
\end{align}
where the function $V$ represents the theoretical relationship between density and speed on a road segment, commonly referred to as the 'fundamental diagram of traffic flow.' This relationship captures the variation in speed with respect to density based on basic intuitive principles, such as the observation that vehicles move slowly in congested traffic, but does not explicitly account for interactions between vehicles. In this work, $V$ is assumed to have the same form as in \cite{KotsialosPapageorgiou}, specifically:
\begin{equation}
\label{eq:fundamentalDiagram}
V(\rho)=v_{max}\exp\left(-\frac{1}{a}\left(\frac{\rho}{\rho_{cr}}\right)^a\right),
\end{equation}
where $v_{max}$ (the maximum speed), $\rho_{cr}$ (the density where the speed starts to decrease considerably) and $a$ (characterising the steepness of this decrease) are positive parameters inherent to every edge of the road network (see \cite[Section 2.1]{vandissel2024paynewhitham} for more details).
The meaning of the conservation law \eqref{eq:density_continuouos} is that the density of cars propagates with the speed given by $v(t,s)$. The speed likewise propagates with the traffic flow, but its variation also depends on other factors given on the right-hand side (RHS) of \eqref{eq:speed_continuouos}. In particular, the first term on the RHS of \eqref{eq:speed_continuouos} drives the speed towards the ideal speed given by the fundamental diagram \eqref{eq:fundamentalDiagram}. The strength of this adaptive behaviour is controlled by the parameter $\nu>0$, which must be calibrated to describe the real-world situation of interest. The second term on the RHS of \eqref{eq:speed_continuouos} represents the ability of drivers to anticipate the traffic conditions ahead and adapt their speed to the change in density. The strength of this ability to anticipate is described by the parameter $C>0$.

Next, an explicit finite difference scheme is used to compute the density and speed at the next time step $k+1$ with respect to the current state (at time step $k$) for each road $e$ uniformly discretised into $N_e$ segments indexed by $i\in \{1,2,\ldots,N_e\}$ (where vehicles travel from segment $i-1$ directly to segment $i$):

\begin{equation}
\label{eq:density_numerical}
\frac{\rho(k+1,i)-\rho(k,i)}{\delta t}+\frac{\rho(k,i)v(k,i)-\rho(k,i-1)v(k,i-1)}{\delta s}=0;
\end{equation}

\begin{equation}
\label{eq:speed_numerical}
\begin{aligned}&\frac{v(k+1,i)-v(k,i)}{\delta t}+\frac{v(k,i)^2-v(k,i-1)^2}{2\delta s}\\
&\quad\quad=\frac{V(\rho(k,i+1))-v(k,i)}{\nu}-\frac{C}{\rho(k,i)+\chi}\frac{\rho(k,i+1)-\rho(k,i)}{\delta s}.
\end{aligned}
\end{equation}
Here $\delta t$ and $\delta s$ are the discretisation grid sizes and $\chi>0$ is a parameter used to avoid blow-ups in the last term of \eqref{eq:speed_numerical}. It may also need to be calibrated to reflect the local situation. Refer to \cite[Section 3]{vandissel2024paynewhitham} for more details about the discretisation technique and the stability conditions for \eqref{eq:density_numerical}-\eqref{eq:speed_numerical}.

\begin{figure}
\begin{center}
\includegraphics[scale=0.36]{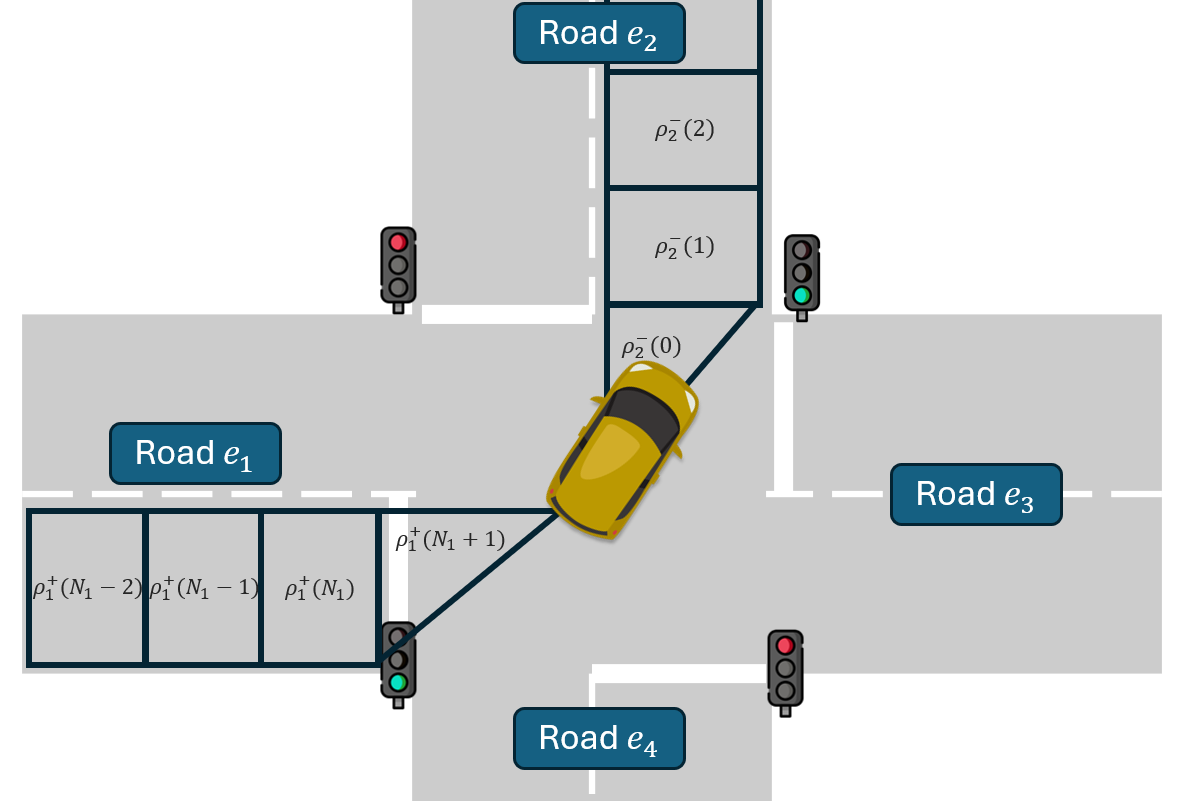}

\caption{An intersection with the virtual initial and final densities $\rho_2^{-}(0)$ and $\rho_{1}^{+}(N_1+1)$, respectively. The signs $\pm$ in the notation are used to help differentiate between the incoming and outgoing sides of each street, in relation to the depicted intersection.}
 \label{fig:intersection1}
\end{center}
\end{figure}


\subsubsection{Coupling at Intersections}\label{sec:intersections}
In the case of a road network, the number of lanes on each street might differ and cars need to adapt to changes in that number by modifying their speed. Therefore, the network version of the Payne-Whitham model should take into account the number of lanes $\ell(e)$ corresponding to each edge $e$ of the graph. The next step is to write equation \eqref{eq:density_numerical} in terms of the traffic flux (i.e. the number of cars passing on all lanes in each instant):
\begin{equation}
\label{eq:flux_numerical}
q(k,i)=\rho(k,i)v(k,i)\ell(e).
\end{equation}
where $\rho(k,i)$ is the density per lane averaged over all lanes. Therefore, equation \eqref{eq:density_numerical} can be rewritten as:
\begin{equation}
\label{eq:density_numerical_flux}
\frac{\rho(k+1,i)-\rho(k,i)}{\delta t}+\frac{1}{\ell(e)}\frac{q(k,i)-q(k,i-1)}{\delta s}=0.
\end{equation}
Further, it can be observed that some quantities needed to perform the iteration in \eqref{eq:density_numerical_flux} and \eqref{eq:speed_numerical} are not available from the previous time step. Namely, in order to compute the density and speed for both ends of a road segment $e$ (parameterised by $\{1,2,\ldots, N_e\}$), one needs to consider some virtual values for $q(k,0)$, $\rho(k,N_e+1)$ and $v(k,0)$.

In the case of a lack of traffic lights, we follow the ideas in \cite{vandissel2024paynewhitham,KotsialosPapageorgiou}, where these virtual values are defined as weighted sums of the traffic values on the other road segments adjacent to the intersection in question. We illustrate these calculations for the particular case of the intersection in Figure \ref{fig:intersection1}:
\begin{align}
\label{eq:density_virtual_initial}
q_{2}^-(k,0) &=\sum_{i=1}^4 q_{i}^+(k,N_{i})\cdot \frac{\omega(e_i,e_2)}{\sum_{j=1}^4 \omega(e_i,e_j)};\\
\label{eq:speed_virtual_initial}
v_{2}^-(k,0) &=\frac{1}{q_{2}^-(k,0)}\sum_{i=1}^4 v_{i}^+(k,N_{i})\cdot q_{i}^+(k,N_{i})\cdot \frac{\omega(e_i,e_2)}{\sum_{j=1}^4 \omega(e_i,e_j)},
\end{align}
where we recall that the flux $q$ is computed by the formula \eqref{eq:flux_numerical}. A similar approach is employed to calculate $\rho_{1}^+(k,N_{1}+1)$ (see \cite[Section 4.1]{vandissel2024paynewhitham}). The turning weights $\omega(e_i,e_j)$ represent the number of cars that choose to continue their trip on road $e_j$, after entering that particular intersection from road $e_i$. These weights are determined by processing real-world traffic data.

\subsubsection{Limiting the Turning Speed} The first improvement in \cite{vandissel2024paynewhitham} to the algorithm in \cite{KotsialosPapageorgiou} concerns the maximum speed of cars changing direction at an intersection. The virtual initial speed $v(k,0)$ is limited according to the speed at the other edges in the intersection and the geometry thereof. We illustrate the procedure using the notation in Figure \ref{fig:intersection1}: the virtual initial speed $v_2^-(k,0)$ must not exceed the following value:
\begin{align}
\label{eq:speed_limit_initial}
&v^{lim}_{2}(k,0)\coloneqq\frac{1}{q_{2}^-(k,0)}\sum_{i=1}^4 v^{max}_{i}\frac{(1-\cos(\widehat{e_i,e_2}))}{2}q_{i}^+(k,N_{i})\cdot \frac{\omega(e_i,e_2)}{\sum_{j=1}^4 \omega(e_i,e_j)},
\end{align}
where $v^{max}_{i}$ is the speed at which cars normally travel when the road $e_i$ is empty and by $\widehat{e_i,e_j}$ we understand the angle between the edges $e_i$ and $e_j$ in the traffic network. See \cite[Section 4.2]{vandissel2024paynewhitham} for more details about this way of limiting speed at intersections.

\subsubsection{Adding Traffic Lights}

The major contribution in \cite{vandissel2024paynewhitham} to the traffic simulation model is introducing traffic signal control for some intersections. The modifications imply both the speed values at the end of the roads controlled by the traffic lights and the weighted sums \eqref{eq:density_virtual_initial}-\eqref{eq:speed_limit_initial}. If the signal is red, the final speed $v(N_e)$ of edge $e$ is set to $0$ and $e$ is not taken as an input in the weighted sums for the virtual density and speed of the other edges (i.e., we set $\omega(e,e_i)=0$ for each road $e_i$).

\subsection{Approximating PDEs on $\RR^2$ with Equations on Embedded Graphs}\label{subsec:approx}

Another application of embedded graphs comes as a follow-up to the convergence results in Section \ref{sec:PDEs}. This time, we approximate solutions of PDEs in a two-dimensional space using a coupled system of one-dimensional equations on a grid contained within $\RR^2$.
In this context, we highlight the results in \cite{Nic240} on the Dirichlet problem in the square $\mathcal{S}=(0,1)\times(0,1)$
\begin{equation}
\label{eq:dirichlet-problem-square}
\begin{cases}
    -\Delta u = f, & \text{on } \mathcal{S};\\
    u = 0, & \text{on }\partial \mathcal{S},
\end{cases}
\end{equation}
where $f$ is a continuous function on $\overline{\mathcal{S}}$.

In order to approximate the solution of the above problem, we consider the traditional equidistant grid in the square, but now seen as a metric graph $\mathcal{G}_h$. Namely, for $N\geq 3$, we set the grid width $h=1/N$ and the digraph $\mathcal{G}_h=(V_h,E_h,\Phi^{\pm}_h,W_h)$, defined as follows (see Figure \ref{fig:discretise_square}):

The set of vertices
\begin{eqnarray*}
    V_h  &\coloneqq& \left\{v_i=\left(v_i^1,v_i^2\right): i\in I^h\right\}\\
    &\coloneqq& \left\{(kh,jh): k,j\in \mathbb{N}\cap [0,N] \right\}\setminus\{ (0,0), (0,Nh), (Nh,0), (Nh,Nh)\}
\end{eqnarray*} 
gets divided into two components, the set of internal and boundary vertices, respectively:
\[\begin{aligned}
V_h^{ext} &\coloneqq \left\{v\in V_h: \,\,v^1\in \{0,N\}\text{ or }v^2\in \{0,N\}\right\};\\
V_h^{int} &\coloneqq V_h\setminus V_h^{ext}.
\end{aligned}\]
The edges connect every two adjacent (in the sense of coordinate number) vertices that are not both exterior nodes, namely
\begin{eqnarray*} E_h &\coloneqq &\left\{e_j:\,j\in J^h\right\}\\
&\coloneqq&\left\{(v_i,v_k): \,\,\left|v_i^1-v_k^1\right|+\left|v_i^2-v_k^2\right|=1, \,\,\text{and } v_i \text{ or } v_k \in V_h^{int} \right\}.
\end{eqnarray*}

\begin{figure}[h]
\centering
 \begin{tikzpicture}[scale=1]
 \draw[-Stealth] (-5,0) -- (1,0) node[right] {$x$};
\draw[-Stealth] (-4,-1) -- (-4,5) node[above] {$y$};
 \foreach \x in {-3,-2,-1}
 \foreach \y in {0,1,2,3}
 \draw[blue,-{Stealth[sep][sep]}] (\x,\y) -- (\x,\y+1);
 \foreach \y in {0,1,2,3}
 \draw (0,0) -- (0,4);
\foreach \y in {1,2,3}
 \foreach \x in {-4,-3,-2,-1}
 \draw[blue,-{Stealth[sep][sep]}] (\x,\y) -- (\x+1,\y);
        \draw (-4,4) -- (0,4);
        
        \foreach \x in {-3,-2,-1}
            \foreach \y in {0,1,2,3,4}
                \filldraw (\x,\y) circle (2pt);

        \foreach \x in {-4,0}
            \foreach \y in {1,2,3}
                \filldraw (\x,\y) circle (2pt);
                
        \draw[decorate,decoration={brace,amplitude=5pt,mirror},thick] (-4.05,-0.2) -- node[below=6pt] {$h$} (-2.95,-0.2);
        \draw[decorate,decoration={brace,amplitude=5pt},thick] (-4.2,0.05) -- node[left=6pt] {$h$} (-4.2,0.95);
        \node[below=3pt] at (0,0) {$1$};
        \node[left=3pt] at (-4,4) {$1$};
    \end{tikzpicture}
    \caption{The discretisation of the square $\mathcal{S} \subset \RR^2$. The edges of the graph $\mathcal{G}_h$ are coloured in blue.}
    \label{fig:discretise_square} 
\end{figure}
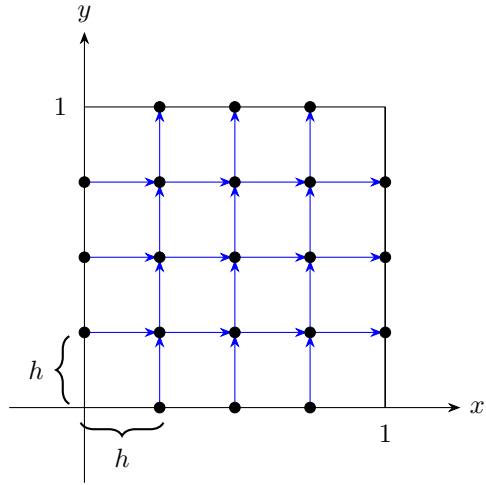

In order to obtain an oriented metric graph, we choose the orientation of the edges of $\mathcal{G}_h$ in the positive direction of each axis (see Figure \ref{fig:discretise_square}). Furthermore, we assume that digraph is unweighted, hence $W_h$ is an identity matrix. This allows us to consider the Dirichlet problem on $\mathcal{G}_h$ corresponding to the Kirchhoff Laplacian $\Delta_{\mathcal{G}_h}$ defined in Subsection \ref{ssLaplaceongraphs},
with Dirichlet boundary conditions at the exterior vertices that are now part of the operator domain:
\begin{equation}
\begin{aligned}
D\left(\Delta_{\mathcal{G}_h}\right)=&\left\{u=\left(u_j\right)_{j\in J^h} \in C(\mathcal{G}_h): u_j \in H^2(e_j),\, \forall e_j \in E_h,\right. \\
& \sum_{j\in J^h}\left(\phi_{h,ij}^{+}-\phi_{h,ij}^{-}\right) u_j^{\prime}\left(v_i\right)=0,\, \forall v_i \in V_h^{int} \text{ and }\\
& \left. u(v_i)=0,\, \forall v_i\in V_h^{ext}\right\}
\end{aligned}
\end{equation}

With this definition, one can prove that the following elliptic equation on graph $\mathcal{G}_h$ is well-posed \cite[Section 2.3]{Nic240}:

\begin{equation}
    \label{eq:dirichlet-problem-grid}
        \Delta_{\mathcal{G}_h} u_h = \frac{1}{2}f\vert_{\mathcal{G}_h}
\end{equation}
In order to state the approximation results, we need to define the $h_h^1$ and $\ell_h^\infty$ norms on the grid $\mathcal{G}_h$:
\[\begin{aligned}
\|u\|_{h_h^1}^2 &\coloneqq \sum_{j\in J^h} \int_{e_j} \left[(u_j)^2+(u_j')^2\right]{\rm d}s;\\
\|u\|_{\ell_h^\infty} &\coloneqq {\rm esssup} \{|u_h(s)|: s\in e_j\in E_h\}.
\end{aligned}
\]
In these norms, one has an approximation result of order $\sqrt{h}$, which is the subject of the main theorem in this section:
\begin{thm}\cite[Cor.~3.6]{Nic240}
Let $f\in W^{1,p}(\mathcal{S})$, with $p>2$, such that $f$ is zero at each corner of $\mathcal{S}$. Then the solution $u$ of \eqref{eq:dirichlet-problem-square} belongs to $H^3(\mathcal{S})$. Furthermore, there exists a constant $C>0$ independent of $h$ and $f$ such that, if $u_h$ is the solution of \eqref{eq:dirichlet-problem-grid}, the following error estimates hold:
\begin{align}
    \|u_h-u\vert_{\mathcal{G}_h}\|_{h_h^1}&\leq C\sqrt{h} \|u\|_{H^3(\mathcal{S})};\\
    \|u_h-u\vert_{\mathcal{G}_h}\|_{\ell^\infty}&\leq C\sqrt{h} \|u\|_{H^3(\mathcal{S})}.
\end{align}
\end{thm}
The meaning of the aforementioned result is that the Dirichlet problem on a square in $\RR^2$ can be regarded as a limit of equations posed on a family of embedded metric graphs, as the embedding progressively encompasses the two-dimensional domain, in what we will later call the ``densification of the graph'', refer to Figure \ref{fig:diagram_up}.

For similar findings on arbitrary domains, interested readers may refer to \cite[Section 5]{Nic240}. Additionally, approximation results for the spectrum of operators on domains in $\RR^2$ with the graph spectrum of embedded networks can be found in \cite{Nic63} and the references therein. 

 
\section{Interdependence of Dynamical Network Models}\label{sec:interdep}

In this paper we have made an attempt to gather into a single overview a vast group of applied problems where dynamics is driven on network-like-structures. At first glance, this idea seems plausible to attain since mathematics operates on such a high level of generality that topics from distant research fields such as education, economy, ecology, or epidemiology, can be tackled with similar tools. Unfortunately, the above statement is true only to a certain extent. When going into details we understand that the peculiarity of each field coerces the development of tailor-made theories that cannot always be easily applied in other disciplines. As a result, even in one applied mathematical field, we develop parallel theories with different basic assumptions (related to state space, definition of solution, etc.).

In this summary, we look into the perspective of major groups of networked models considered. Having unified the language, in Section \ref{sec:prelimi}, we start tracing both differences and similarities in the models' foundations. Our goal is to sketch a map that depicts a world of dynamical systems on networks with cities/vertices being the major groups of mathematical objects considered in the field and roads/edges giving their interrelations. The text below is nothing else than a pocket guide with a description of large cities as well as known (but not necessarily main) paths in this world. Figure \ref{fig:multi-layer} can serve us as a large-scale map while Figures \ref{fig:diagram_up} and \ref{fig:diagram_down} as car GPS navigation.   

\subsection{Description of the world of networked dynamical systems} \label{subsec:world}

Let us start the construction of our world from scratch, namely from the definition of the networked model. In order to apply the network paradigm, one needs to observe a group of objects that interact with one another. These interactions constitute the backbone of every networked model. In this study we mostly indicate the difference in the role of the interaction's initiator and its recipient, hence the object that structures the model is a digraph. The topology of a network may be related to the physical location of the objects involved, and then we talk about a graph which is embedded into $\RR^n$ space. On the other hand, it may be based on any other common feature of the objects, and then the interaction does not need to have a physical representation. In order to distinguish these two cases we build a two-layered network with vertices being mathematical tools related to unembedded graphs at the lower layer and tools related to embedded networks at the upper one. 

Let $V_i$, $i=1,\ldots 5$ be vertices at the lower layer representing respectively combinatorial graphs ($V_1$), operators on combinatorial graphs ($V_2$), spectral theory of operators on combinatorial graphs $(V_3)$, PDEs on edges of metric graphs $(V_4)$ and ODEs in vertices of combinatorial graphs $(V_5)$. In the upper layer, vertices are denoted by $V_i^e$, $i=1,\ldots, 6$. In the case $i=1,\ldots, 5$ they are embedded counterparts of notions at the lower layer, but there is one additional vertex $V_6$ representing PDEs/ODEs on vector space in which the system is embedded. In Figure \ref{fig:multi-layer} all vertices at both levels are presented together with a raw sketch of relations between them. In order to understand detailed interrelations one has to go further into the picture presenting layers separately,  Figure \ref{fig:diagram_down} for the lower and Figure \ref{fig:diagram_up} for the upper layer.

We hypothesize that embedding can be regarded as an additional assumption of the model that gives rise to additional relations between entities. One can observe that Figure \ref{fig:diagram_up} includes all the elements (vertices and edges) from Figure \ref{fig:diagram_down} but is enriched with some new properties. 

Let us now characterize all inscribed relations in more detail.\newline

\textbf{Paths $V_1 - V_2 - V_3 - V_i - V_1$ (with variant $V_1 -V_2 - V_4 - V_1$), $i=4,5$.} The notion of a network comes from a set theory but Section \ref{sec:combinat} clearly shows that combinatorial digraph ($V_1$) also provides information about dynamics. From a mathematical perspective, this knowledge can be captured using operators which describe the network (e.g. incidence $\Phi^{\pm}$ or weight $W$ matrices) or some network characteristics described either in one moment in time (e.g. when the process is at equilibrium point) or over several moments (which permits describing the temporary state of the system).

In order to understand the process, better one can define more sophisticated operators on digraphs ($V_2$). They can be of different natures, either purely related to the network structure 
(see \eqref{eq:unwundiradjacency} -- \eqref{eq:laplac-beltr}) or providing additional information 
(\eqref{eq:flow-op}, \eqref{eq:LaplC_dom}, \eqref{eq:LaplC}, \eqref{eq:Lapl-Lp-dom}, \eqref{eq:Lapl-Lp}); defined in vertices (eg. \eqref{eq:adv} -- 
\eqref{eq:laplac-beltr}) or on edges (eg. \eqref{eq:flow-op}, \eqref{eq:LaplC_dom}, \eqref{eq:LaplC}, \eqref{eq:Lapl-Lp-dom}, \eqref{eq:Lapl-Lp}), etc. The properties of operators such as boundedness, self-adjointness, positivity, irreducibility, etc.,\ provide additional information about phenomena such as the connectivity of groups, symmetry, or periodicity of interactions. Furthermore, one can define a functional, called a graph measure, acting from the network either to the subset of $\RR$ in the case of global measures (e.g. reciprocity \eqref{eq:reciprocity} or modularity \eqref{eq:modularity}), or to the subset of $\RR^n$/$\RR^m$ in the case of local measures (such as Finn cycling index \eqref{eq:singlenode}
). All of these notions are examined in Section \ref{sec:combinat}.

Among all algebraic methods that serve to examine operators on networks, we emphasize spectral theory ($V_3$) which has become a standard tool in analyses of dynamical properties. On the one hand, people compare the spectrum of network operators and define graph measures based on them (eg. eigenvector centrality, PageRank). From another perspective, having network operators defined, one can consider both linear and non-linear abstract Cauchy problems (ACP) generated by them, moving smoothly from algebraic methods to pure analysis. Depending on the type of an operator, we arrive at ODEs defined in vertices, $V_4$ (Section \ref{sec:ODEs}), or PDEs defined on edges, $V_5$ (Section \ref{sec:PDEs}). 
 
 In order to return to the starting point we indicate that a well-known method for recovering the weights of initial combinatorial digraphs ($V_1$) from the observed phenomenon is examination of long-time behaviours of the process given by the Abstract Cauchy Problem or its point evaluation in time.\newline
 
\indent \textbf{Path $V_4 - V_5 - V_4$.}
The classification of dynamics in vertices and on edges is not straightforward, which has already been shown in Subsection \ref{subsec:virus2}. If we allow for inhomogeneity of objects in vertices, we introduce an additional variable $x$ extending the dimension of the vertex $v_i$, for example into dimension one in the simplest case. Then the functional space above the vertex can be defined and one can consider PDEs in vertices instead of ODEs. In this case an edge having a head in $v_j$ and a tail in $v_i$ may, for example, inform about the flow of mass from one point at a locally one-dimensional metric space in the vertex $v_i$ to another point in another vertex $v_j$, \cite{BFN2016}. This can be obviously more complex taking for instance non-local transfers, \cite{BGS2011}. This reasoning can be easily extended to vertices of arbitrary dimension.

In order to state the relation of dynamics in locally one-dimensional vertices to corresponding dynamics defined on the edges of a metric graph, one needs to ensure that it is possible to build a metric graph based on network relations, proposed by incidence matrices $\Phi^{\pm}$. Figuratively speaking, if we, for instance, associate vertex $v_i$ with interval $[0,l_i]$, then we need to “glue” the ends of this interval with the ends of other intervals related to the vertices that are adjacent to $v_i$. The endpoints of the intervals constitute new vertices and the intervals themselves become the edges of a newly defined metric graph. If such a procedure holds, we say that the problem is graph realisable, \cite{BF:15}. However, it is always possible to represent PDE on metric graph edges by the PDEs in the vertices' framework.

Additionally, it is sometimes possible to relate PDEs on metric graph's edges to ODEs in vertices by methods of aggregation such as asymptotic state lumping \cite{BFN2016}.  This process can also be interpreted as the aggregation of the process in the micro-scale into the macro model.\newline

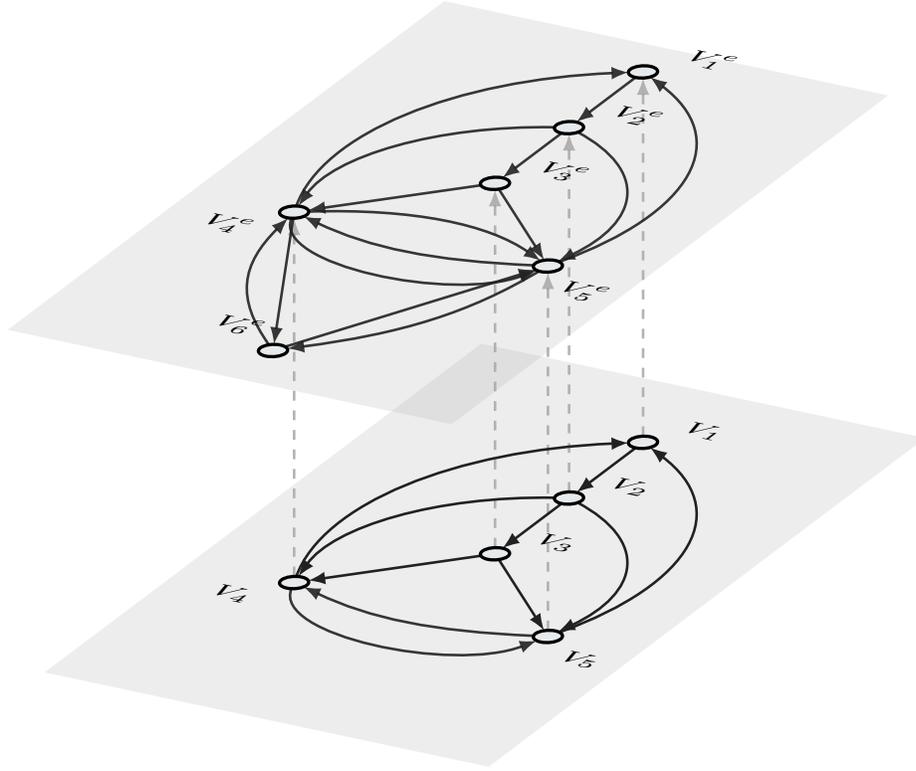
\begin{figure}[h!]
\SetCoordinates[yLength=1,xLength=1.4]
\resizebox{\textwidth}{!}{%
\begin{tikzpicture}[multilayer=3d]
\SetLayerDistance{4}
\Plane[x=-1.8,y=-.5,width=3.5,height=5.9,color={rgb, 255:red, 169; green, 169; blue, 169},NoBorder,layer=2,opacity=.2]
\Plane[x=-1.8,y=0,width=3.5,height=5.9,
color={rgb, 255:red, 169; green, 169; blue, 169},NoBorder,layer=1,opacity=.2]
\Vertices{multilayer_v2.csv}
\Edges{multilayer_e2.csv}
\end{tikzpicture}}
\caption{Two-layer digraph representing interrelations between different network models considered in the paper. Vertices $V_i$ and $V_i^e$ represent respectively unembedded notions and their embedded counterparts, namely combinatorial graphs (for $i=1$), operators on combinatorial graphs (for $i=2$), spectral theory of operators on combinatorial graphs (for $i=3$), PDEs on edges of metric graphs (for $i=4$) and ODEs in vertices of combinatorial graphs (for $i=5$). Vertex $V_6^e$ in the upper layer represents PDEs/ODEs on the whole embedded vector space, which appears only in the embedded case.}
\label{fig:multi-layer}
\end{figure}

\textbf{Path $V_i - V_i^e$, $i=1,\ldots, 5$.} Every unembedded mathematical tool considered in this study can be transformed into a tool applied to an embedded network. One needs to embed the graph structure into a larger vector space and then define the object which now has additional information about its embedding; see the definition of the embedded metric graph in Section \ref{sec:met_net}. It is worth underlining that information about the localisation in the space can be one of the parameters in a dynamical process changing its intensity, like that presented in Section \ref{subsec:traffic}.\newline

\textbf{Path $V_6^e - V_5^e - V_6^e - V_4^e - V_5^e$.} The application of embedded objects allows us to introduce geometry into the study, see the planar embedding of a digraph/metric graph in Subsection \ref{sec:met_net}, instead of giving only the weighted relation between groups of objects. Furthermore, as presented in Subsection \ref{subsec:traffic}, in many cases the dynamics depend on the geometry and consequently, PDEs on embedded edges and ODEs/PDEs in embedded vertices may constitute more complex dynamics compared to their unembedded counterparts.

{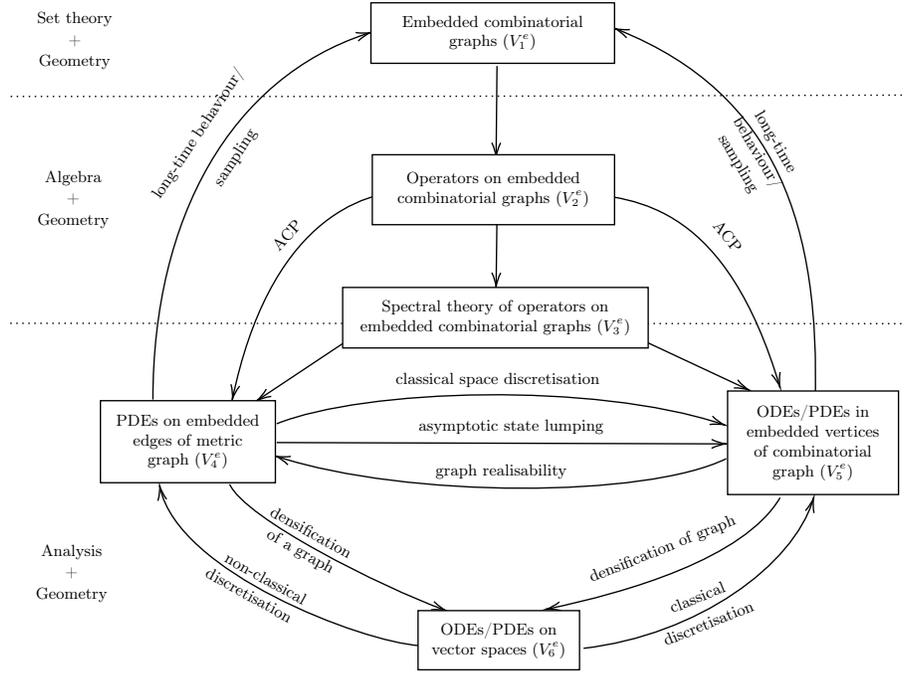
\begin{figure}[h!]
\begin{center}


\tikzset{every picture/.style={line width=0.75pt}} 

\resizebox{\textwidth}{!}{%

\tikzset{every picture/.style={line width=0.75pt}} 

\begin{tikzpicture}[x=0.75pt,y=0.75pt,yscale=-1,xscale=1]

\draw [color={rgb, 255:red, 0; green, 0; blue, 0 }  ,draw opacity=1 ] [dash pattern={on 0.84pt off 2.51pt}]  (-1.67,79.24) -- (267,79.24) -- (670,78.2) ;
\draw [color={rgb, 255:red, 0; green, 0; blue, 0 }  ,draw opacity=1 ] [dash pattern={on 0.84pt off 2.51pt}]  (-2,247.91) -- (668,248.2) ;
\draw [color={rgb, 255:red, 0; green, 0; blue, 0 }  ,draw opacity=1 ]   (262.93,32.58) .. controls (162.15,65.68) and (104,168.97) .. (104,304.6) ;
\draw [shift={(266,31.6)}, rotate = 162.69] [color={rgb, 255:red, 0; green, 0; blue, 0 }  ,draw opacity=1 ][line width=0.75]    (10.93,-3.29) .. controls (6.95,-1.4) and (3.31,-0.3) .. (0,0) .. controls (3.31,0.3) and (6.95,1.4) .. (10.93,3.29)   ;
\draw [color={rgb, 255:red, 0; green, 0; blue, 0 }  ,draw opacity=1 ]   (450.61,30.04) .. controls (568.96,77.93) and (599.96,176.92) .. (596,307.6) ;
\draw [shift={(447,28.6)}, rotate = 21.31] [color={rgb, 255:red, 0; green, 0; blue, 0 }  ,draw opacity=1 ][line width=0.75]    (10.93,-3.29) .. controls (6.95,-1.4) and (3.31,-0.3) .. (0,0) .. controls (3.31,0.3) and (6.95,1.4) .. (10.93,3.29)   ;
\draw    (296.33,149.24) .. controls (258,149.4) and (190.33,173.27) .. (163,304.6) ;
\draw [shift={(163,304.6)}, rotate = 281.76] [color={rgb, 255:red, 0; green, 0; blue, 0 }  ][line width=0.75]    (10.93,-3.29) .. controls (6.95,-1.4) and (3.31,-0.3) .. (0,0) .. controls (3.31,0.3) and (6.95,1.4) .. (10.93,3.29)   ;
\draw    (252.67,257.91) -- (182.67,304.49) ;
\draw [shift={(181,305.6)}, rotate = 326.36] [color={rgb, 255:red, 0; green, 0; blue, 0 }  ][line width=0.75]    (10.93,-3.29) .. controls (6.95,-1.4) and (3.31,-0.3) .. (0,0) .. controls (3.31,0.3) and (6.95,1.4) .. (10.93,3.29)   ;
\draw    (465,259.24) -- (545.19,296.75) ;
\draw [shift={(547,297.6)}, rotate = 205.07] [color={rgb, 255:red, 0; green, 0; blue, 0 }  ][line width=0.75]    (10.93,-3.29) .. controls (6.95,-1.4) and (3.31,-0.3) .. (0,0) .. controls (3.31,0.3) and (6.95,1.4) .. (10.93,3.29)   ;
\draw    (418,150.91) .. controls (466,154.4) and (530,161.6) .. (570,295.6) ;
\draw [shift={(570,295.6)}, rotate = 253.38] [color={rgb, 255:red, 0; green, 0; blue, 0 }  ][line width=0.75]    (10.93,-3.29) .. controls (6.95,-1.4) and (3.31,-0.3) .. (0,0) .. controls (3.31,0.3) and (6.95,1.4) .. (10.93,3.29)   ;
\draw [color={rgb, 255:red, 0; green, 0; blue, 0 }  ,draw opacity=1 ]   (161,367.6) .. controls (172.88,393.34) and (272.97,443.6) .. (320.57,461.08) ;
\draw [shift={(322,461.6)}, rotate = 199.86] [color={rgb, 255:red, 0; green, 0; blue, 0 }  ,draw opacity=1 ][line width=0.75]    (10.93,-3.29) .. controls (6.95,-1.4) and (3.31,-0.3) .. (0,0) .. controls (3.31,0.3) and (6.95,1.4) .. (10.93,3.29)   ;
\draw [color={rgb, 255:red, 0; green, 0; blue, 0 }  ,draw opacity=1 ]   (302,487.6) .. controls (258.44,484.63) and (123.08,432.98) .. (109.38,370.5) ;
\draw [shift={(109,368.6)}, rotate = 79.85] [color={rgb, 255:red, 0; green, 0; blue, 0 }  ,draw opacity=1 ][line width=0.75]    (10.93,-3.29) .. controls (6.95,-1.4) and (3.31,-0.3) .. (0,0) .. controls (3.31,0.3) and (6.95,1.4) .. (10.93,3.29)   ;
\draw [color={rgb, 255:red, 0; green, 0; blue, 0 }  ,draw opacity=1 ]   (424,489.6) .. controls (509.14,479.7) and (579.58,427.98) .. (594.56,379.08) ;
\draw [shift={(595,377.6)}, rotate = 105.84] [color={rgb, 255:red, 0; green, 0; blue, 0 }  ,draw opacity=1 ][line width=0.75]    (10.93,-3.29) .. controls (6.95,-1.4) and (3.31,-0.3) .. (0,0) .. controls (3.31,0.3) and (6.95,1.4) .. (10.93,3.29)   ;
\draw [color={rgb, 255:red, 0; green, 0; blue, 0 }  ,draw opacity=1 ]   (570,377.6) .. controls (542.94,423.14) and (455.77,448.1) .. (397.75,460.24) ;
\draw [shift={(396,460.6)}, rotate = 348.31] [color={rgb, 255:red, 0; green, 0; blue, 0 }  ,draw opacity=1 ][line width=0.75]    (10.93,-3.29) .. controls (6.95,-1.4) and (3.31,-0.3) .. (0,0) .. controls (3.31,0.3) and (6.95,1.4) .. (10.93,3.29)   ;
\draw    (359.67,56.58) -- (359.02,122.58) ;
\draw [shift={(359,124.58)}, rotate = 270.56] [color={rgb, 255:red, 0; green, 0; blue, 0 }  ][line width=0.75]    (10.93,-3.29) .. controls (6.95,-1.4) and (3.31,-0.3) .. (0,0) .. controls (3.31,0.3) and (6.95,1.4) .. (10.93,3.29)   ;
\draw    (359.67,167.24) -- (359.02,219.24) ;
\draw [shift={(359,221.24)}, rotate = 270.71] [color={rgb, 255:red, 0; green, 0; blue, 0 }  ][line width=0.75]    (10.93,-3.29) .. controls (6.95,-1.4) and (3.31,-0.3) .. (0,0) .. controls (3.31,0.3) and (6.95,1.4) .. (10.93,3.29)   ;
\draw    (197.87,347.53) .. controls (293.23,377.94) and (469.28,378.14) .. (530,348.6) ;
\draw [shift={(195,346.6)}, rotate = 18.31] [color={rgb, 255:red, 0; green, 0; blue, 0 }  ][line width=0.75]    (10.93,-3.29) .. controls (6.95,-1.4) and (3.31,-0.3) .. (0,0) .. controls (3.31,0.3) and (6.95,1.4) .. (10.93,3.29)   ;
\draw    (196,336.6) -- (528,337.59) ;
\draw [shift={(530,337.6)}, rotate = 180.17] [color={rgb, 255:red, 0; green, 0; blue, 0 }  ][line width=0.75]    (10.93,-3.29) .. controls (6.95,-1.4) and (3.31,-0.3) .. (0,0) .. controls (3.31,0.3) and (6.95,1.4) .. (10.93,3.29)   ;
\draw [color={rgb, 255:red, 0; green, 0; blue, 0 }  ,draw opacity=1 ]   (526.15,322.64) .. controls (431.22,291.44) and (262.11,296.4) .. (196,322.6) ;
\draw [shift={(529,323.6)}, rotate = 198.87] [color={rgb, 255:red, 0; green, 0; blue, 0 }  ,draw opacity=1 ][line width=0.75]    (10.93,-3.29) .. controls (6.95,-1.4) and (3.31,-0.3) .. (0,0) .. controls (3.31,0.3) and (6.95,1.4) .. (10.93,3.29)   ;

\draw (14,15) node [anchor=north west][inner sep=0.75pt]   [align=left] {\begin{minipage}[lt]{45.58pt}\setlength\topsep{0pt}
\begin{center}
{\footnotesize Set theory}\\{\footnotesize + Geometry}
\end{center}

\end{minipage}};
\draw (13,133) node [anchor=north west][inner sep=0.75pt]   [align=left] {\begin{minipage}[lt]{45.58pt}\setlength\topsep{0pt}
\begin{center}
{\footnotesize Algebra}\\{\footnotesize + Geometry}
\end{center}

\end{minipage}};
\draw  [fill={rgb, 255:red, 255; green, 255; blue, 255 }  ,fill opacity=1 ]  (267,122.59) -- (447,122.59) -- (447,174.59) -- (267,174.59) -- cycle  ;
\draw (357,148.59) node  [font=\footnotesize] [align=left] {\begin{minipage}[lt]{119.68pt}\setlength\topsep{0pt}
\begin{center}
Operators on\textcolor[rgb]{0.56,0.07,1}{ } embedded\\combinatorial graphs ($\displaystyle V_{2}^{e}$)
\end{center}

\end{minipage}};
\draw (99.67,151.36) node [anchor=north west][inner sep=0.75pt]  [color={rgb, 255:red, 0; green, 0; blue, 0 }  ,opacity=1 ,rotate=-303.41] [align=left] {{\footnotesize long-time behaviour/}};
\draw  [fill={rgb, 255:red, 255; green, 255; blue, 255 }  ,fill opacity=1 ]  (65,305.47) -- (195,305.47) -- (195,366.47) -- (65,366.47) -- cycle  ;
\draw (130,335.97) node  [font=\footnotesize] [align=left] {\begin{minipage}[lt]{85.68pt}\setlength\topsep{0pt}
\begin{center}
PDEs on\textcolor[rgb]{0.56,0.07,1}{ } embedded\\edges of metric \\graph ($\displaystyle V_{4}^{e})$
\end{center}

\end{minipage}};
\draw (145.87,142.54) node [anchor=north west][inner sep=0.75pt]  [color={rgb, 255:red, 0; green, 0; blue, 0 }  ,opacity=1 ,rotate=-303.01] [align=left] {{\footnotesize sampling}};
\draw (11.67,411) node [anchor=north west][inner sep=0.75pt]   [align=left] {\begin{minipage}[lt]{45.58pt}\setlength\topsep{0pt}
\begin{center}
{\footnotesize Analysis }\\{\footnotesize + Geometry}
\end{center}

\end{minipage}};
\draw  [fill={rgb, 255:red, 255; green, 255; blue, 255 }  ,fill opacity=1 ]  (531,298.26) -- (658,298.26) -- (658,375.26) -- (531,375.26) -- cycle  ;
\draw (594.5,336.76) node  [font=\footnotesize] [align=left] {\begin{minipage}[lt]{83.64pt}\setlength\topsep{0pt}
\begin{center}
ODEs/PDEs in\textcolor[rgb]{0.56,0.07,1}{ }\\embedded\textcolor[rgb]{0.56,0.07,1}{ } vertices\\ of combinatorial \\graph $\displaystyle \left( V_{5}^{e}\right)$
\end{center}

\end{minipage}};
\draw  [color={rgb, 255:red, 0; green, 0; blue, 0 }  ,draw opacity=1 ][fill={rgb, 255:red, 255; green, 255; blue, 255 }  ,fill opacity=1 ]  (301,461.58) -- (421,461.58) -- (421,505.58) -- (301,505.58) -- cycle  ;
\draw (361,483.58) node  [font=\footnotesize] [align=left] {\begin{minipage}[lt]{78.88pt}\setlength\topsep{0pt}
\begin{center}
ODEs/PDEs on \\vector spaces ($\displaystyle V_{6}^{e})$
\end{center}

\end{minipage}};
\draw  [fill={rgb, 255:red, 255; green, 255; blue, 255 }  ,fill opacity=1 ]  (245,221.26) -- (472,221.26) -- (472,266.26) -- (245,266.26) -- cycle  ;
\draw (358.5,243.76) node  [font=\footnotesize] [align=left] {\begin{minipage}[lt]{151.64pt}\setlength\topsep{0pt}
\begin{center}
Spectral theory of operators on\textcolor[rgb]{0.56,0.07,1}{ }\\embedded\textcolor[rgb]{0.56,0.07,1}{ } combinatorial graphs ($\displaystyle V_{3}^{e})$
\end{center}

\end{minipage}};
\draw (545.33,64.31) node [anchor=north west][inner sep=0.75pt]  [color={rgb, 255:red, 0; green, 0; blue, 0 }  ,opacity=1 ,rotate=-59.37] [align=left] {\begin{minipage}[lt]{76.66pt}\setlength\topsep{0pt}
\begin{center}
{\footnotesize long-time behaviour/}\\{\footnotesize  sampling }
\end{center}

\end{minipage}};
\draw (189.61,189.01) node [anchor=north west][inner sep=0.75pt]  [rotate=-312.86] [align=left] {{\footnotesize ACP}};
\draw (526.2,168.25) node [anchor=north west][inner sep=0.75pt]  [rotate=-50.74] [align=left] {{\footnotesize ACP}};
\draw (197.59,374.34) node [anchor=north west][inner sep=0.75pt]  [color={rgb, 255:red, 0; green, 0; blue, 0 }  ,opacity=1 ,rotate=-29] [align=left] {\begin{minipage}[lt]{50.35pt}\setlength\topsep{0pt}
\begin{center}
{\footnotesize densification of a graph}
\end{center}

\end{minipage}};
\draw (160,415.66) node [anchor=north west][inner sep=0.75pt]  [color={rgb, 255:red, 0; green, 0; blue, 0 }  ,opacity=1 ,rotate=-29] [align=left] {{\footnotesize non-classical}};
\draw (145.77,426.78) node [anchor=north west][inner sep=0.75pt]  [color={rgb, 255:red, 0; green, 0; blue, 0 }  ,opacity=1 ,rotate=-29] [align=left] {{\footnotesize discretisation}};
\draw (485.04,455) node [anchor=north west][inner sep=0.75pt]  [color={rgb, 255:red, 0; green, 0; blue, 0 }  ,opacity=1 ,rotate=-338] [align=left] {{\footnotesize classical}};
\draw (480.24,479.2) node [anchor=north west][inner sep=0.75pt]  [color={rgb, 255:red, 0; green, 0; blue, 0 }  ,opacity=1 ,rotate=-338] [align=left] {{\footnotesize discretisation}};
\draw (425.12,432.05) node [anchor=north west][inner sep=0.75pt]  [color={rgb, 255:red, 0; green, 0; blue, 0 }  ,opacity=1 ,rotate=-338] [align=left] {{\footnotesize densification of graph}};
\draw (282.97,283.65) node [anchor=north west][inner sep=0.75pt]  [color={rgb, 255:red, 0; green, 0; blue, 0 }  ,opacity=1 ] [align=left] {{\footnotesize classical space discretisation}};
\draw (312.67,351.58) node [anchor=north west][inner sep=0.75pt]   [align=left] {{\footnotesize graph realisability}};
\draw (300,318.24) node [anchor=north west][inner sep=0.75pt]   [align=left] {{\footnotesize asymptotic state lumping}};
\draw  [fill={rgb, 255:red, 255; green, 255; blue, 255 }  ,fill opacity=1 ]  (266,9.59) -- (447,9.59) -- (447,54.59) -- (266,54.59) -- cycle  ;
\draw (356.5,32.09) node  [font=\footnotesize] [align=left] {\begin{minipage}[lt]{120.36pt}\setlength\topsep{0pt}
\begin{center}
Embedded combinatorial graphs ($\displaystyle V_{1}^{e}$) 
\end{center}

\end{minipage}};

\end{tikzpicture}

}


\caption{The upper layer in the two-layer digraph given in Figure \ref{fig:multi-layer}. Diagram illustrates interrelations between different mathematical tools considered for embedded networks. Different paths in the graph have been presented in Section \ref{subsec:world}.} 
\label{fig:diagram_up}
\end{center}
\end{figure}}
{\begin{figure}[ht!]
\begin{center}


\tikzset{every picture/.style={line width=0.75pt}} 

\resizebox{\textwidth}{!}{%

\tikzset{every picture/.style={line width=0.75pt}} 

\begin{tikzpicture}[x=0.75pt,y=0.75pt,yscale=-1,xscale=1]

\draw [color={rgb, 255:red, 0; green, 0; blue, 0 }  ,draw opacity=1 ] [dash pattern={on 0.84pt off 2.51pt}]  (-1.67,79.24) -- (267,79.24) -- (670,78.2) ;
\draw [color={rgb, 255:red, 0; green, 0; blue, 0 }  ,draw opacity=1 ] [dash pattern={on 0.84pt off 2.51pt}]  (-2,247.91) -- (668,248.2) ;
\draw [color={rgb, 255:red, 0; green, 0; blue, 0 }  ,draw opacity=1 ]   (262.93,32.58) .. controls (162.15,65.68) and (104,168.97) .. (104,304.6) ;
\draw [shift={(266,31.6)}, rotate = 162.69] [color={rgb, 255:red, 0; green, 0; blue, 0 }  ,draw opacity=1 ][line width=0.75]    (10.93,-3.29) .. controls (6.95,-1.4) and (3.31,-0.3) .. (0,0) .. controls (3.31,0.3) and (6.95,1.4) .. (10.93,3.29)   ;
\draw [color={rgb, 255:red, 0; green, 0; blue, 0 }  ,draw opacity=1 ]   (450.61,30.04) .. controls (568.96,77.93) and (599.96,176.92) .. (596,307.6) ;
\draw [shift={(447,28.6)}, rotate = 21.31] [color={rgb, 255:red, 0; green, 0; blue, 0 }  ,draw opacity=1 ][line width=0.75]    (10.93,-3.29) .. controls (6.95,-1.4) and (3.31,-0.3) .. (0,0) .. controls (3.31,0.3) and (6.95,1.4) .. (10.93,3.29)   ;
\draw    (296.33,149.24) .. controls (258,149.4) and (190.33,173.27) .. (163,304.6) ;
\draw [shift={(163,304.6)}, rotate = 281.76] [color={rgb, 255:red, 0; green, 0; blue, 0 }  ][line width=0.75]    (10.93,-3.29) .. controls (6.95,-1.4) and (3.31,-0.3) .. (0,0) .. controls (3.31,0.3) and (6.95,1.4) .. (10.93,3.29)   ;
\draw    (252.67,257.91) -- (182.67,304.49) ;
\draw [shift={(181,305.6)}, rotate = 326.36] [color={rgb, 255:red, 0; green, 0; blue, 0 }  ][line width=0.75]    (10.93,-3.29) .. controls (6.95,-1.4) and (3.31,-0.3) .. (0,0) .. controls (3.31,0.3) and (6.95,1.4) .. (10.93,3.29)   ;
\draw    (465,259.24) -- (545.19,296.75) ;
\draw [shift={(547,297.6)}, rotate = 205.07] [color={rgb, 255:red, 0; green, 0; blue, 0 }  ][line width=0.75]    (10.93,-3.29) .. controls (6.95,-1.4) and (3.31,-0.3) .. (0,0) .. controls (3.31,0.3) and (6.95,1.4) .. (10.93,3.29)   ;
\draw    (418,150.91) .. controls (466,154.4) and (530,161.6) .. (570,295.6) ;
\draw [shift={(570,295.6)}, rotate = 253.38] [color={rgb, 255:red, 0; green, 0; blue, 0 }  ][line width=0.75]    (10.93,-3.29) .. controls (6.95,-1.4) and (3.31,-0.3) .. (0,0) .. controls (3.31,0.3) and (6.95,1.4) .. (10.93,3.29)   ;
\draw    (359.67,56.58) -- (359.02,122.58) ;
\draw [shift={(359,124.58)}, rotate = 270.56] [color={rgb, 255:red, 0; green, 0; blue, 0 }  ][line width=0.75]    (10.93,-3.29) .. controls (6.95,-1.4) and (3.31,-0.3) .. (0,0) .. controls (3.31,0.3) and (6.95,1.4) .. (10.93,3.29)   ;
\draw    (359.67,167.24) -- (359.02,219.24) ;
\draw [shift={(359,221.24)}, rotate = 270.71] [color={rgb, 255:red, 0; green, 0; blue, 0 }  ][line width=0.75]    (10.93,-3.29) .. controls (6.95,-1.4) and (3.31,-0.3) .. (0,0) .. controls (3.31,0.3) and (6.95,1.4) .. (10.93,3.29)   ;
\draw    (197.87,347.53) .. controls (293.23,377.94) and (469.28,378.14) .. (530,348.6) ;
\draw [shift={(195,346.6)}, rotate = 18.31] [color={rgb, 255:red, 0; green, 0; blue, 0 }  ][line width=0.75]    (10.93,-3.29) .. controls (6.95,-1.4) and (3.31,-0.3) .. (0,0) .. controls (3.31,0.3) and (6.95,1.4) .. (10.93,3.29)   ;
\draw    (196,336.6) -- (528,337.59) ;
\draw [shift={(530,337.6)}, rotate = 180.17] [color={rgb, 255:red, 0; green, 0; blue, 0 }  ][line width=0.75]    (10.93,-3.29) .. controls (6.95,-1.4) and (3.31,-0.3) .. (0,0) .. controls (3.31,0.3) and (6.95,1.4) .. (10.93,3.29)   ;

\draw (16,16) node [anchor=north west][inner sep=0.75pt]   [align=left] {\begin{minipage}[lt]{39.92pt}\setlength\topsep{0pt}
\begin{center}
{\footnotesize Set theory}
\end{center}

\end{minipage}};
\draw (16,135) node [anchor=north west][inner sep=0.75pt]   [align=left] {\begin{minipage}[lt]{30.85pt}\setlength\topsep{0pt}
\begin{center}
{\footnotesize Algebra}
\end{center}

\end{minipage}};
\draw  [fill={rgb, 255:red, 255; green, 255; blue, 255 }  ,fill opacity=1 ]  (267,122.59) -- (447,122.59) -- (447,174.59) -- (267,174.59) -- cycle  ;
\draw (357,148.59) node  [font=\footnotesize] [align=left] {\begin{minipage}[lt]{119.68pt}\setlength\topsep{0pt}
\begin{center}
Operators on\textcolor[rgb]{0.56,0.07,1}{ } embedded\\combinatorial graphs ($\displaystyle V_{2}^{e}$)
\end{center}

\end{minipage}};
\draw (99.67,151.36) node [anchor=north west][inner sep=0.75pt]  [color={rgb, 255:red, 0; green, 0; blue, 0 }  ,opacity=1 ,rotate=-303.41] [align=left] {{\footnotesize long-time behaviour/}};
\draw  [fill={rgb, 255:red, 255; green, 255; blue, 255 }  ,fill opacity=1 ]  (65,305.47) -- (195,305.47) -- (195,366.47) -- (65,366.47) -- cycle  ;
\draw (130,335.97) node  [font=\footnotesize] [align=left] {\begin{minipage}[lt]{85.68pt}\setlength\topsep{0pt}
\begin{center}
PDEs on\textcolor[rgb]{0.56,0.07,1}{ } embedded\\edges of metric \\graph ($\displaystyle V_{4}^{e})$
\end{center}

\end{minipage}};
\draw (145.87,142.54) node [anchor=north west][inner sep=0.75pt]  [color={rgb, 255:red, 0; green, 0; blue, 0 }  ,opacity=1 ,rotate=-303.01] [align=left] {{\footnotesize sampling}};
\draw (15.67,385) node [anchor=north west][inner sep=0.75pt]   [align=left] {\begin{minipage}[lt]{35.38pt}\setlength\topsep{0pt}
\begin{center}
{\footnotesize Analysis }
\end{center}

\end{minipage}};
\draw  [fill={rgb, 255:red, 255; green, 255; blue, 255 }  ,fill opacity=1 ]  (531,298.26) -- (658,298.26) -- (658,375.26) -- (531,375.26) -- cycle  ;
\draw (594.5,336.76) node  [font=\footnotesize] [align=left] {\begin{minipage}[lt]{83.64pt}\setlength\topsep{0pt}
\begin{center}
ODEs/PDEs in\textcolor[rgb]{0.56,0.07,1}{ }\\embedded\textcolor[rgb]{0.56,0.07,1}{ } vertices\\ of combinatorial \\graph $\displaystyle \left( V_{5}^{e}\right)$
\end{center}

\end{minipage}};
\draw  [fill={rgb, 255:red, 255; green, 255; blue, 255 }  ,fill opacity=1 ]  (245,221.26) -- (472,221.26) -- (472,266.26) -- (245,266.26) -- cycle  ;
\draw (358.5,243.76) node  [font=\footnotesize] [align=left] {\begin{minipage}[lt]{151.64pt}\setlength\topsep{0pt}
\begin{center}
Spectral theory of operators on\textcolor[rgb]{0.56,0.07,1}{ }\\embedded\textcolor[rgb]{0.56,0.07,1}{ } combinatorial graphs ($\displaystyle V_{3}^{e})$
\end{center}

\end{minipage}};
\draw (545.33,64.31) node [anchor=north west][inner sep=0.75pt]  [color={rgb, 255:red, 0; green, 0; blue, 0 }  ,opacity=1 ,rotate=-59.37] [align=left] {\begin{minipage}[lt]{76.66pt}\setlength\topsep{0pt}
\begin{center}
{\footnotesize long-time behaviour/}\\{\footnotesize  sampling }
\end{center}

\end{minipage}};
\draw (189.61,189.01) node [anchor=north west][inner sep=0.75pt]  [rotate=-312.86] [align=left] {{\footnotesize ACP}};
\draw (526.2,168.25) node [anchor=north west][inner sep=0.75pt]  [rotate=-50.74] [align=left] {{\footnotesize ACP}};
\draw (312.67,351.58) node [anchor=north west][inner sep=0.75pt]   [align=left] {{\footnotesize graph realisability}};
\draw (300,318.24) node [anchor=north west][inner sep=0.75pt]   [align=left] {{\footnotesize asymptotic state lumping}};
\draw  [fill={rgb, 255:red, 255; green, 255; blue, 255 }  ,fill opacity=1 ]  (266,9.59) -- (447,9.59) -- (447,54.59) -- (266,54.59) -- cycle  ;
\draw (356.5,32.09) node  [font=\footnotesize] [align=left] {\begin{minipage}[lt]{120.36pt}\setlength\topsep{0pt}
\begin{center}
Embedded combinatorial graphs ($\displaystyle V_{1}^{e}$) 
\end{center}

\end{minipage}};

\end{tikzpicture}

}


\caption{The lower layer in the two-layer digraph given in Figure \ref{fig:multi-layer}. The diagram illustrates interrelations between different mathematical tools considered for unembedded networks. Different paths in the graph have been presented in Section \ref{subsec:world}.} 
\label{fig:diagram_down}
\end{center}
\end{figure}
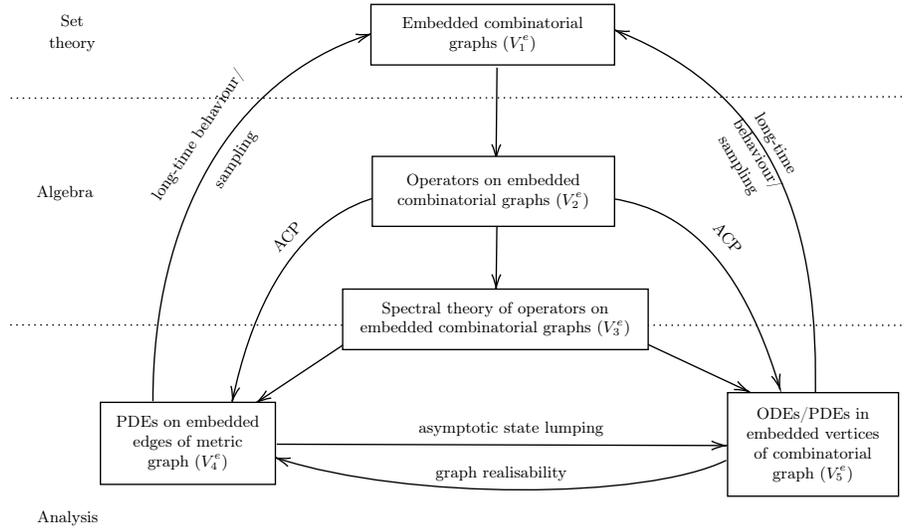}

Another advantage is the fact that by adding the geometry, one can consider the relation between the considered tool and the reference vector space in which this tool is embedded. Namely, a network of PDEs ($V_4^e$) can be regarded as a specific discretisation of the system of PDEs defined on the whole vector space $V_6^e$, as shown in Subsection \ref{subsec:approx}. Starting from a planar system one can choose for example standard discretization to obtain a set of vertices in which the system evolves, hence the model on vertices $V_5^e$. On the other hand, using non-standard discretisation by describing evolution along locally one-dimensional spaces we arrive at a metric graph model. Finally, discretizing the latter by considering the dynamical process at vertices of a metric graph only, we return to the combinatorial graph system.

\section{Conclusions}
To the best of our knowledge, this is the first paper that encompasses such a broad perspective on the world of dynamical systems on networks. Apart from analytical tools based on differential equations, we include statistical and numerical network methods closely used with real applied problems. Noticing the common aims and unifying a framework for further study is the first step on the way 
to building coherent mathematical theory. 

The breadth of this paper' perspective shows how the diversity of real-world problems related to dynamical systems on networks can be answered with a broad range of mathematical frameworks. Going too much into detail risks losing the bird's-eye view, but being too general we miss details that are crucial to understand interrelations between objects. Consequently, in this study we have chosen to present a selected group of topics for the broad field, going into the details to some extent to underline the interrelations.

It is not surprising that the world that we describe is convoluted, multilayered, and unclear at first glance. This follows from the complexity of the problems concerned. Nevertheless, the possibility of building first bridges between theories represented graphically in Figures \ref{fig:diagram_up}-\ref{fig:diagram_down}-\ref{fig:multi-layer} is a promising outcome and can serve as a stimulus for further study.

\bibliographystyle{plain}
\bibliography{joint_bib2}
\end{document}